\documentclass[11pt]{amsart}
\usepackage{amssymb}
\usepackage{latexsym}
\usepackage{graphics}
\usepackage{graphicx}
\usepackage{epstopdf}
\usepackage[all]{xy}
\usepackage{verbatim}
\input epsf

\def\TSG{{\mathrm{TSG}_+}}
\def\fix{{\mathrm{fix}}}
\def\Aut{{\mathrm{Aut}}}

\newtheorem*{fact}{Fact}

\newtheorem{cor}{Corollary}
\newtheorem{lemma}{Lemma}

\newtheorem*{edge}{Edge Embedding Lemma}
\newtheorem*{subgroup}{Subgroup Lemma}
\newtheorem*{subgroupcor}{Subgroup Corollary}

\newtheorem*{fixed vertex}{Fixed Vertex Property}
\newtheorem*{smith}{Smith Theory}

\def\Z{{\mathbb Z}}

\def\R{{\mathbb R}}

\def\a{{\alpha}}
\def\b{{\beta}}

\def\f{{\phi}}

\def\ep{{\varepsilon}}
\def\vf{{\varphi}}

\def\TSG{{\mathrm{TSG_+}}}
\def\Aut{{\mathrm{Aut}}}
\def\Diff{{\mathrm{Diff_+}}}
\def\fix{{\mathrm{fix}}}

\def\so{{\mathrm{SO}}}
\def\lcm{{\mathrm{lcm}}}

\newcommand{\x}{\times}

\newtheorem*{iso}{\textbf{Isometry Theorem}}

\newtheorem*{aut}{\textbf{Automorphism Theorem}}
\newtheorem*{orbit}{\textbf{Orbits Lemma}}
\newtheorem*{disjoint}{\textbf{Disjoint Fixed Points Lemma}}

\newtheorem*{them1}{\textbf{Theorem 1}}
\newtheorem*{them2}{\textbf{Theorem 2}}

\newtheorem{theorem}{Theorem}

\theoremstyle{definition}

\theoremstyle{remark}
\newtheorem*{remark}{Remark}

\begin{document}
\title{Topological symmetry groups of complete bipartite graphs}
\author[Hake]{Kathleen Hake}
\author[Mellor]{Blake Mellor}
\author[Pittluck]{Matt Pittluck}



\address{Department of Mathematics, UC Santa Barbara}
\email{khake@math.ucsb.edu}
\address{Department of Mathematics, Loyola Marymount University, Los Angeles, CA 90045, USA}
\email{blake.mellor@lmu.edu}

\date \today

\thanks{This research was supported in part by NSF Grant DMS-0905687.}

\begin{abstract}

The symmetries of complex molecular structures can be modeled by the {\em topological symmetry group} of the underlying embedded graph.  It is therefore important to understand which topological symmetry groups can be realized by particular abstract graphs.  This question has been answered for complete graphs \cite{fmn3}; it is natural next to consider complete bipartite graphs.  In previous work we classified the complete bipartite graphs that can realize topological symmetry groups isomorphic to $A_4$, $S_4$ or $A_5$ \cite{me}; in this paper we determine which complete bipartite graphs have an embedding in $S^3$ whose topological symmetry group is isomorphic to $\Z_m$, $D_m$, $\Z_r \x \Z_s$ or $(\Z_r \x \Z_s) \ltimes \Z_2$. 

\end{abstract}

\maketitle

\section{Introduction}

Chemists have long used the symmetries of a molecule to predict some of its chemical properties. For small molecules, it is enough to consider the rigid symmetries, such as rotations and reflections.  Increasingly, however, chemists are dealing with long, flexible molecules (such as DNA), for which the group of rigid symmetries is no longer sufficient. To help understand the symmetries of these more complex molecules, Jon Simon introduced the {\em topological symmetry group} \cite{si}. Molecules are often modeled as graphs, where vertices represent atoms and edges represent bonds. Although the motivation for studying topological symmetry groups arose from looking at symmetries of molecules, we can consider the topological symmetry group of any embedded graph. 

We consider an abstract graph $\gamma$ with automorphism group $\Aut(\gamma)$, and let $\Gamma$ be an embedding of $\gamma$ in $S^3$.  The {\em topological symmetry group} of $\Gamma$, denoted $\mathrm{TSG}(\Gamma)$, is the subgroup of $\Aut(\gamma)$ induced by diffeomorphisms of the pair $(S^3, \Gamma)$.  The {\em orientation preserving topological symmetry group} of $\Gamma$, denoted $\TSG(\Gamma)$, is the subgroup of $\Aut(\gamma)$ induced by orientation preserving diffeomorphisms of the pair $(S^3, \Gamma)$.  In this paper we are only concerned with $\TSG(\Gamma)$, so we will refer to it as simply the {\em topological symmetry group}.

It has long been known that every finite group can be realized as $\Aut(\gamma)$ for some graph $\gamma$ \cite{fr}.  However, this is {\em not} true for topological symmetry groups.  Results of Flapan, Naimi, Pommersheim and Tamvakis \cite{fnpt}, in combination with the Geometrization Conjecture \cite{mf}, show that any topological symmetry group of an embedding of a 3-connected graph is isomorphic to a finite subgroup of $\so(4)$.  However, their results do not give any information as to which graphs can be used to realize any particular group.  The first results along these lines have been for the family of complete graphs $K_n$.  Flapan, Naimi and Tamvakis \cite{fnt} classified the groups which could be realized as the topological symmetry group for an embedding of a complete graph; subsequently, Flapan, Naimi, Yoshizawa and the second author determined exactly which complete graphs had embeddings that realized each group \cite{fmn2, fmn3}.

In this paper we turn to another well-known family of graphs, the {\it complete bipartite graphs} $K_{n,n}$. Unlike the complete graphs, where only some of the subgroups of $\so(4)$ are realizable as topological symmetry groups, {\em any} finite subgroup of $\so(4)$ can be realized as the topological symmetry group of an embedding of some $K_{n,n}$ \cite{fnpt}.  So the complete bipartite graphs are a natural family of graphs to investigate in order to better understand the full range of possible topological symmetry groups.  The finite subgroups of $\so(4)$ have been classified and they can all be described as quotients of products of cyclic groups $\Z_m$, dihedral groups $D_m$, and the symmetry groups of the regular polyhedra ($A_4$, $S_4$ and $A_5$) \cite{du}.  Previously, the second author determined which complete bipartite graphs have embeddings whose topological symmetry groups are isomorphic to $A_4$, $S_4$ or $A_5$ \cite{me}.  In this paper we consider the groups $\Z_m$, $D_m$, $\Z_r \times \Z_s$ and $(\Z_r \times \Z_s) \ltimes \Z_2$.  The results are summarized in the following theorems:

\begin{theorem}\label{T:cyclic}
Let $n>2$. There exists an embedding, $\Gamma$, of $K_{n,n}$ in $S^{3}$ such that $\TSG(\Gamma)=H$ for $H=\mathbb{Z}_{m}$ or  $D_{m}$ if and only if one of the following conditions hold:
	\begin{enumerate}
	\item $n\equiv 0,1,2 \pmod{m}$,
	\item $n\equiv 0 \pmod{\frac{m}{2}}$ when $m$ is even,
	\item $n\equiv 2 \pmod{\frac{m}{2}}$ when $m$ is even and $4|m$.
\end{enumerate}
\end{theorem}

\begin{theorem}\label{T:product}
Let $n>2$. There exists an embedding, $\Gamma$, of $K_{n,n}$ in $S^{3}$ such that $H \subseteq\TSG(\Gamma)$ for $H = \mathbb{Z}_{r}\times\mathbb{Z}_{s}$ or $(\mathbb{Z}_{r}\times\mathbb{Z}_{s})\ltimes\mathbb{Z}_{2}$, where $r \vert s$, if and only if one of the following conditions hold:
	\begin{enumerate}	
	\item $n\equiv 0 \pmod {s}$,
	\item $n\equiv 2 \pmod {2s}$ when $r=2$,
	\item $n\equiv s+2 \pmod {2s}$ when $4\vert s$, and $r=2$,
	\item $n\equiv 2 \pmod {2s}$ when $r=4$.
	\end{enumerate}
Moreover, in each of the above cases, we can construct embeddings $\Gamma$ where $\TSG(\Gamma) = H$ except in the following cases, which are still open: \begin{itemize}
	\item $K_{ls, ls}$, when $1 \leq l < 2r$, $H = \mathbb{Z}_{r}\times\mathbb{Z}_{s}$ or $(\mathbb{Z}_{r}\times\mathbb{Z}_{s})\ltimes\mathbb{Z}_{2}$
	\item $K_{6,6}$, when $H = (\Z_2 \x \Z_4) \ltimes \Z_2$
	\item $K_{10,10}$, when $H = (\Z_4 \x \Z_4) \ltimes \Z_2$
\end{itemize}
\end{theorem}

\begin{remark}
Since, for any $r$ and $s$, $\mathbb{Z}_{r}\times\mathbb{Z}_{s}\cong\mathbb{Z}_{\gcd(r,s)}\times\mathbb{Z}_{\lcm(r,s)}$, it is easiest to assume that $r \vert s$, so that $\gcd(r,s) = r$ and $\lcm(r,s) = s$.  In general, Theorem \ref{T:product} could be written with $r$ replaced by $\gcd(r,s)$ and $s$ replaced by $\lcm(r,s)$ in each of the conditions.
\end{remark}

\section{Background}

\subsection{Prior results} 
In this section we gather together prior results that we will refer to throughout this paper.  We first consider results that allow us to prove that certain groups {\it cannot} be realized as a topological symmetry group for a particular graph.  The following well-known fact for complete bipartite graphs restricts how automorphisms of the graph can act on the vertices.

\begin{fact}
\label{auts}
Let $\f $ be a permutation of the vertices of $K_{n,n}$.  Let $V$ and $W$ denote the two sets of $n$ independent vertices.  Then $\f$ is an automorphism of $K_{n,n}$ if and only if $\f$ either interchanges $V$ and $W$ or setwise fixes each of $V$ and $W$.
\end{fact}

The following result about finite order homeomorphisms of $S^3$ is a special case of a well-known result of P. A. Smith.

\begin{smith} \cite{sm}
Let $h$ be a non-trivial finite order homeomorphism of $S^3$.  If $h$ is orientation preserving, then $\fix(h)$ is either the empty set or is homeomorphic to $S^1$.  If $h$ is orientation reversing, then $\fix(h)$ is homeomorphic to either $S^0$ or $S^2$.
\end{smith}

The Isometry Theorem allows us to assume that the elements of $\TSG(\Gamma)$ are orientation-preserving isometries -- i.e. either rotations (whose fixed point sets are geodesic circles) or glide rotations (with no fixed points).

\begin{iso}{\cite{fnpt}}
Let $\Gamma$ be an embedded 3-connected graph, and let $H=\TSG(\Gamma)$. Then $\Gamma$ can be re-embedded as $\Gamma'$ such that $H \subseteq \TSG(\Gamma')$ and $\TSG(\Gamma')$ is induced by an isomorphic subgroup of $\so(4)$. \end{iso}

The Automorphism Theorem \cite{flmpv} tells us which automorphisms of $K_{n,n}$ can be realized as an orientation-preserving diffeomorphism of $(S^{3}, \Gamma)$, for some embedding $\Gamma$ of $K_{n,n}$.

\begin{aut}{\cite{flmpv}}
Let $n>2$ and let $\varphi$ be an order $r$ automorphism of a complete bipartite graph $K_{n,n}$ with vertex sets $V$ and $W$.  There is an embedding $\Gamma$ of $K_{n,n}$ in $S^3$ with an orientation preserving diffeomorphism $h$ of $(S^3,\Gamma)$ inducing $\varphi$ if and only if all vertices are in $r$-cycles except for the fixed vertices and exceptional cycles explicitly mentioned below (up to interchanging $V$ and $W$):

\begin{enumerate} 
\item There are no fixed vertices or exceptional cycles.

\item $V$ contains one or more fixed vertices.

\item $V$ and $W$ each contain at most 2 fixed vertices.

\item $j|r$ and $V$ contains some $j$-cycles.

\item $r=\mathrm{lcm}(j,k)$, and $V$ contains some $j$-cycles and $k$-cycles.

\item $r=\mathrm{lcm}(j,k)$, and $V$ contains some $j$-cycles and $W$ contains some $k$-cycles.

\item $V$ and $W$ each contain one 2-cycle.

\item $\frac{r}{2}$ is odd, $V$ and $W$ each contain one 2-cycle, and $V$ contains some $\frac{r}{2}$-cycles.

\item $\varphi(V)=W$ and $V\cup W$ contains one 4-cycle.

\end{enumerate}
 
 \end{aut}
 
 \begin{orbit}{\cite{cfo}}
Suppose $\alpha$ and $\beta$ are commuting automorphisms of a finite set $V$. Then $\beta$ takes $\alpha$-orbits to $\alpha$-orbits of the same length.
\end{orbit} 

\begin{disjoint}{\cite{cfo}}
Suppose g, h $\in \Diff(S^{3})$ such that $\langle g,h \rangle = \mathbb{Z}_{r}\times\mathbb{Z}_{s}$ is not cyclic or equal to $D_{2}$. Then $\fix(g)$ and $\fix(h)$ are disjoint.
\end{disjoint}

The following lemmas will be useful when we construct an embedding of $K_{n,n}$ in $S^{3}$ that realizes a particular automorphism $\varphi$. The Edge Embedding Lemma will help us extend an embedding of the vertices of $K_{n,n}$ to an embedding of the edges with the same symmetries.  The Subgroup Lemma and Subgroup Corollary allow us to re-embed the graph to realize a smaller group of symmetries.

\begin{edge}{\cite{fmn2}}
Let G be a finite subgroup of $\Diff(S^{3})$, and let $\gamma$ be a graph whose vertices are embedded in $S^{3}$ as a set V which is invariant under G such that G induces a faithful action on $\gamma$. Suppose that adjacent pairs of vertices in V satisfy the following hypotheses:

\begin{enumerate}
\item If a pair is pointwise fixed by non-trivial elements h,g $\in$ G, then $\fix(h) = \fix(g)$.
\item For each pair $\{v,w\}$ in the fixed point set C of some non-trivial element of G, there is an arc $A_{vw} \subseteq$ C bounded by {v,w} whose interior is disjoint from V and from any other such arc $A_{v'w'}$.
\item If a point in the interior of some $A_{vw}$ or a pair $\{v,w\}$ bounding some $A_{vw}$ is setwise invariant under some f $\in$ G, then $f(A_{vw})=A_{vw}$.
\item If a pair is interchanged by some g $\in$ G, then the subgraph of $\gamma$ whose vertices are pointwise fixed by g can be embedded in a proper subset of a circle.
\item If a pair is interchanged by some g $\in$ G, then $\fix(g)$ is non-empty, and $\fix(h)\ne \fix(g)$ if $h\ne g$.
\end{enumerate}

\noindent Then the embedding of the vertices of $\gamma$ can be extended to the edges of $\gamma$ in $S^{3}$ such that the resulting embedding of $\gamma$ is setwise invariant under G.
\end{edge}

\begin{subgroup}\cite{fmn1}
Let $\Gamma$ be an embedding of a 3-connected graph $\gamma$ in $S^{3}$, and let $H \subseteq \TSG(\Gamma)$. Let $\varepsilon_{1}$,...,$\varepsilon_{n}$ be edges of $\gamma$ embedded in $\Gamma$ as $e_{1}$,...,$e_{n}$.  Let $\langle e_i \rangle_H$ denote the orbit of edge $e_i$ under the action of $H$, and let $\langle \ep_i \rangle_H$ denote the orbit of $\ep_i$ under the action of the subgroup of automorphisms of $\gamma$ induced by $H$.

Now suppose that $\langle e_{1} \rangle _{H}$,...,$\langle e_{n} \rangle _{H}$ are distinct and that any automorphism of $\gamma$ which fixes $\varepsilon_{1}$ pointwise, and fixes each $\langle \varepsilon_{i} \rangle_{H}$ setwise, also pointwise fixes a subgraph of $\gamma$ which cannot be embedded in $S^{1}$. Then there is an embedding $\Gamma'$ of $\gamma$ such that $\TSG(\Gamma') = H$.
\end{subgroup}

\begin{subgroupcor}\cite{fmn1} 
Let $\Gamma$ be an embedding of a 3-connected graph in $S^3$.  Suppose that $\Gamma$ contains an edge $e$ which is not pointwise fixed by any non-trivial element of $\TSG(\Gamma)$.  Then for every $H \subseteq \TSG(\Gamma)$, there is an embedding $\Gamma'$ of $\Gamma$ with $H = \TSG(\Gamma')$.
\end{subgroupcor}

\subsection{Motions in $\so(4)$}
In this section we will describe the structure of $\so(4)$ and lay out some facts we will need later in the paper.  For more details, see Du Val \cite{du} and Conway and Smith \cite{cs}.  We will also describe some particular subgroups of $\so(4)$ which we will use to realize topological symmetry groups.

Algebraically, $\so(4)$ is the group of $4 \x 4$ real matrices with determinant 1.  Geometrically, an element of $\so(4)$ is an orientation-preserving rigid motion of $\R^4$ that fixes the origin; we are then interested in the induced motion on the unit sphere $S^3$.  There are two kinds of motions in $\so(4)$: \medskip

A {\em rotation} fixes a plane $A$ through the origin in $\R^4$, and rotates the orthogonal plane $B$ through the origin by some angle $\alpha$.  Depending on the context, the {\it axis} of the rotation denotes either the plane $A$ or the geodesic circle where $A$ intersects $S^3$; we say that we rotate by an angle $\alpha$ about this axis.  Note that, unless $\alpha = \pi$, $B$ and $A$ are the only invariant planes (i.e. the only planes mapped to themselves). \smallskip

A {\em glide rotation} only fixes the origin in $\R^4$ (and so has no fixed points in $S^3$).  Any glide rotation has a pair of mutually orthogonal planes $A$ and $B$ which are invariant, meaning that each plane is rotated onto itself.  In general, $A$ is rotated by an angle $\beta$ and $B$ is rotated by an angle $\alpha$.  The intersections of $A$ and $B$ with $S^3$ are a pair of linked geodesic circles.
The glide rotation can be viewed as the composition of two (commuting) rotations: one by an angle $\alpha$ about $A$, and the other by an angle $\beta$ about $B$.

If the angles $\alpha$ and $\beta$ are not equal in magnitude, the glide rotation $g$ has a {\it unique} pair of invariant planes.  However, if $\alpha = \pm \beta$, then we say the glide rotation $g$ is {\em isoclinic}, and there are infinitely many pairs of invariant planes. For any vector ${\bf v}$ in $\R^4$, the plane spanned by ${\bf v}$ and $g({\bf v})$ is an invariant plane. The isoclinic motions fall into two subgroups, the {\it left-isoclinic} motions (where $\alpha = \beta$) and {\it right-isoclinic} motions (where $\alpha = -\beta$); the intersection of these subgroups is just the identity and the central inversion (multiplication by $-1$).  Every left-isoclinic motion commutes with every right-isoclinic motion, and vice versa.  Every element of $\so(4)$ can be represented as a product of a left-isoclinic motion and a right-isoclinic motion \cite{cs, meb}.\medskip

We use these facts to prove some lemmas which will be useful in the proof of Theorem \ref{T:product}.

\begin{lemma}\label{L:commute}
If $g$ and $h$ are commuting motions in $\so(4)$ (so $gh = hg$), then there is a pair of orthogonal planes $A$ and $B$ which are invariant under {\it both} $g$ and $h$.
\end{lemma}
\begin{proof}
Since every element of $\so(4)$ has at least one pair of invariant orthogonal planes, let $A$ and $B$ be a pair of orthogonal planes that are invariant under $g$.  Then $gh(A) = hg(A) = h(A)$, so $h(A)$ (and, similarly, $h(B)$) is also invariant under $g$. By a change of basis, we may assume that $A$ and $B$ are the $xy$-plane and $zw$-plane in $xyzw$-space, respectively, so $g$ is one of the following matrices:

$$g = \left[\begin{matrix} R_\alpha & 0\\ 0 & R_\beta \end{matrix} \right] {\rm \ or\ } \left[\begin{matrix} MR_\alpha & 0\\ 0 & MR_\beta \end{matrix} \right],$$ \medskip
$${\rm where\ } R_\alpha = \left[ \begin{matrix} \cos \alpha & -\sin\alpha  \\ \sin\alpha & \cos \alpha \end{matrix}\right] {\rm \ and\ } M = \left[ \begin{matrix} 0 & 1  \\ 1 & 0 \end{matrix}\right]$$ \medskip

When $g$ is the matrix on the right, then it is a rotation of order 2; so we first consider the special case when $g$ is a rotation of order 2 and redefine $A$ to be the axis of rotation.  Then, for any $a \in A$, $gh(a) = hg(a) = h(a)$, so $g$ fixes the plane $h(A)$ pointwise.  Hence $h(A) = A$.  Similarly, if $b \in B$, $gh(b) = hg(b) = h(-b) = -h(b)$, so $g$ rotates $h(B)$ by an angle $\pi$.  Hence $h(B) = B$.  So the planes $A$ and $B$ are invariant under both $g$ and $h$. Similarly, if $h$ is a rotation of order 2, we are done.  So from now on, we assume $g$ and $h$ are {\it not} rotations of order 2.

We now consider the case when $g$ is {\it not} an isoclinic glide rotation, so $\alpha \neq \pm \beta$.  Then $A$ and $B$ are the only invariant pair of planes for $g$, so $h$ must either map each plane to itself, or interchange them. If $h(A) = A$ and $h(B) = B$, then we're done; so suppose that $h(A) = B$ and $h(B) = A$. Then $h$ is represented by one of the $4 \times 4$ matrices below:

$$h = \left[\begin{matrix} 0 & R_\gamma \\ R_\delta & 0 \end{matrix} \right] {\rm \ or\ } \left[\begin{matrix} 0 & MR_\gamma \\ MR_\delta & 0 \end{matrix} \right]$$ \medskip

But now, an easy computation shows that $gh = hg$ only when $\alpha = \beta$ (if $h$ is the matrix on the left) or $\alpha = -\beta$ (if $h$ is the matrix on the right).  Since $g$ is not isoclinic, this is a contradiction, so $A$ and $B$ must also be invariant planes for $h$.

Similarly, if $h$ is not isoclinic, we are done.  Moreover, since $g$ and $h$ both commute with $gh$, we are done if $gh$ is not isoclinic.  So now suppose that $g$, $h$ and $gh$ are all isoclinic. Then they must all be left (or all right) isoclinic. Without loss of generality, we suppose they are all left isoclinic. Then $\alpha = \beta$ in the matrix for $g$, and $h$ has the form \cite{meb}:
$$h = \left[\begin{matrix} a & -b & -c & -d \\ b & a & -d & c\\ c & d & a & -b\\ d & -c & b & a \end{matrix} \right], \quad a^2 + b^2 + c^2 + d^2 = 1$$
But then a direct computation shows that $gh = hg$ exactly when $c = d = 0$, so $h$ is also a glide rotation about planes $A$ and $B$.
\end{proof} \medskip

\begin{cor}\label{C:commute}
Suppose $H$ is a subgroup of $\so(4)$ which is isomorphic to $\Z_r \x \Z_s$, where $\lcm(r,s) > 2$.  So $H = \langle g, h | g^r = h^s = 1, gh = hg \rangle$. Then there are two completely orthogonal planes $A$ and $B$ such that $g$ is a combination of a rotation by $\frac{2\pi}{a}$ around $A$ and a rotation by $\frac{2\pi}{b}$ around $B$, with $\lcm(a,b) = r$, and $h$ is a combination of a rotation by $\frac{2\pi}{c}$ around $A$ and a rotation by $\frac{2\pi}{d}$ around $B$, with $\lcm(c,d) = s$.
\end{cor}
\begin{proof}
By Lemma \ref{L:commute}, there must be a pair of completely orthogonal planes $A$ and $B$ which are invariant under both $g$ and $h$. Hence, as in the proof of Lemma \ref{L:commute}, after a change of basis $g$ and $h$ must have one of the forms below:
$$\left[\begin{matrix} R_\alpha & 0\\ 0 & R_\beta \end{matrix} \right] {\rm \ or\ } \left[\begin{matrix} MR_\alpha & 0\\ 0 & MR_\beta \end{matrix} \right]$$ \medskip
Since $\lcm(r,s) > 2$, at least one of $g$ or $h$ is not a rotation of order 2. So at least one preserves the orientations of the planes $A$ and $B$; since they commute, they must both preserve the orientations. Combined with the fact that $g^r = h^s = 1$, $g$ and $h$ must have matrices:
$$g = \left[\begin{matrix} R_\frac{2\pi}{a} & 0\\ 0 & R_\frac{2\pi}{b} \end{matrix} \right] {\rm \ and\ } h = \left[\begin{matrix} R_\frac{2\pi}{c} & 0\\ 0 & R_\frac{2\pi}{d} \end{matrix} \right]$$ \medskip
where $\lcm(a,b) = r$ and $\lcm(c,d) = s$, as desired.
\end{proof} \\

\section{Cyclic and Dihedral Groups}\label{S:cyclic}

If $\Gamma$ is an embedding of $K_{n,n}$ such that $\Z_m \subseteq \TSG(\Gamma)$, then it must have an automorphism of order $m$. The Automorphism Theorem tells us when this is possible.

\begin{lemma} \label{L:cyclic1}
Let $n \ge 2$. If $K_{n,n}$ has an embedding $\Gamma$ such that $\mathbb{Z}_{m}\subseteq\TSG(\Gamma)$, then either \begin{enumerate}
	\item $n\equiv 0,1,2 \pmod{m}$,
	\item $n \equiv 0 \pmod{\frac{m}{2}}$, where $m$ is even, or  
	\item $n\equiv 2 \pmod{\frac{m}{2}}$, where $4|m$.
\end{enumerate}
\end{lemma}
\begin{proof}
Let $\varphi$ be a generator of $\mathbb{Z}_{m}\subseteq\TSG(\Gamma)$, so $\varphi$ is an automorphism of $K_{n,n}$ of order $m$ realized by a orientation-preserving symmetry of $\Gamma$. From the Automorphism Theorem, we have nine cases. In case (1), $\vf$ has only $m$-cycles, so $m \vert 2n$.  Then either $m \vert n$ or $m$ is even and $\frac{m}{2} \vert n$.  Hence either $n \equiv 0 \pmod{m}$ or $m$ is even and $n \equiv 0 \pmod{\frac{m}{2}}$.  In cases (2), (4) and (5), $\vf(V) = V$ and $W$ contains only $m$-cycles, so $n \equiv 0 \pmod{m}$. In cases (3), (7) and (8), $m \vert (n-1)$ or $m \vert (n-2)$, so $n \equiv 1$ or $2 \pmod{m}$.

In case (6), $m = \lcm(j,k)$, $V$ contains some $j$-cycles and $W$ contains some $k$-cycles. Then $n = aj + bm$ for some $a,b \in \mathbb{Z}$ and also $n = ck +dm$ for some $c,d \in \mathbb{Z}$. Thus $aj + bm = ck +dm$, which implies that $dm-bm = aj - ck$. Thus $m(d-b) = aj-ck$, meaning that $m\vert (aj-ck)$. Since $m =\lcm(j,k)$ this means that $j\mid m$ and thus $j\mid (aj-ck)$. So $j\mid ck$, i.e. $ck$ is a multiple of $j$. Also, $ck$ is a multiple of $k$, and thus $ck$ is a common multiple of $j$ and $k$. Since $m$ = lcm($j,k$) and $ck$ is a common multiple of $j$ and $k$, then we have that $m \mid ck$. Hence $m \vert (ck+dm)$, so $n \equiv 0 \pmod{m}$.

Finally, in case (9), $\vf(V) = W$ and $V\cup W$ contains a 4-cycle.  So $m \vert (2n-4)$ and $4 \vert m$.  This means $\frac{m}{2} \vert (n-2)$. So $n \equiv 2 \pmod{\frac{m}{2}}$, with $4 \vert m$.
 \end{proof}\medskip

Now we will prove that, under each of the conditions of Lemma \ref{L:cyclic1}, we can find embeddings of $K_{n,n}$ whose topological symmetry group is $\Z_m$ or $D_m$. Recall that $\Z_m = \langle g \vert g^m = 1\rangle$ and $D_m = \langle g, \varphi \vert g^m = \varphi^2 = 1, \varphi g = g^{-1} \varphi\rangle$.  First we will show how, for each value of $n$ given in Lemma \ref{L:cyclic1}, we can find a group of motions $G$ isomorphic to $D_m$ and an embedding of the vertices of $K_{n,n}$ so that the action of $G$ fixes the vertices setwise. Next we will use the Edge Embedding Lemma to extend the embedding of the vertices to the edges to get an embedding $\Gamma$ of $K_{n,n}$ such that $D_m\subseteq \TSG(\Gamma)$. Finally we will use the Subgroup Lemma to show that we can find another embedding $\Gamma'$ of $K_{n,n}$ such that $\TSG(\Gamma') = D_m$, and the Subgroup Corollary to show there is yet another embedding $\Gamma''$ such that $\TSG(\Gamma'') = \Z_m$.

In our proofs, we will use the following subgroups of $\so(4)$.  Let $A$ be a plane in $\R^4$ and $B$ be its orthogonal complement, and let $C$ be a plane spanned by a vector in $A$ and a vector in $B$. We will let $X$, $Y$ and $Z$ denotes the intersections with $S^3$ of planes $A$, $B$ and $C$, respectively.

\begin{itemize}
\item Let $G_{1}$ be generated by a rotation $g$ around $A$ by $\frac{2\pi}{m}$, and a rotation $\varphi$ around $C$ by $\pi$. Then $g$ has order $m$ and $\varphi$ has order 2, and $G_{1}\cong D_{m}$.

\item Let $G_{2}$ be generated by a rotation $\varphi$ around $C$ by $\pi$, and a glide rotation $h$, which is the combination of a rotation by $\frac{4\pi}{m}$ around $A$ and a rotation by $\frac{2\pi}{m}$ around $B$. Then $h$ has order $m$ and $\varphi$ has order 2, and $G_{2}\cong D_{m}$.

\item Let $G_{3}$ be generated by a rotation $\varphi$ around $C$ by $\pi$, and a glide rotation $j$, which is the combination of a rotation by $\frac{\pi}{2}$ around $A$ and a rotation by $\frac{2\pi}{m}$ around $B$. Then $j$ has order $m$ and $\varphi$ has order 2, and $G_{3}\cong D_{m}$.

\end{itemize}

\begin{lemma} \label{L:n=0,1,2}
If $n \geq 3$, $n \equiv 0, 1, 2 \pmod{m}$ and $H = \Z_m$ or $D_m$, then there exists an embedding, $\Gamma$, of $K_{n,n}$ in $S^{3}$ such that $\TSG(\Gamma)=H$.
\end{lemma}

\begin{proof}
We will first suppose that $H = D_m$, and then use the Subgroup Corollary to show there are also embeddings for $H = \Z_m$.  Since $n\equiv 0, 1, 2 \pmod{m}$, then $2n=2mk + 2\ep$ for some $k \in \mathbb{Z}$, where $\ep = 0, 1, 2$. We will use the group of motions $G_{1}$. Pick a small ball, $M$, such that for each non-trivial $h \in G_1$, $h(M) \cap M = \emptyset$ (i.e. $G_1$ acts freely on $M$).  Pick $k$ points, $p_{1}, \ldots, p_{k}$, in $M$. Each $p_{i}$ has an orbit of size $2m$. Embed $2mk$ of the vertices of $V \cup W$ so that vertices of $V$ are embedded as the points $g^{i}(p_{j})$ and vertices of $W$ are embedded as the points $\varphi g^{i}(p_{j})$. This means that $g(V) = V$ and $\vf(V) = W$. If $\ep = 0$, we have embedded all the vertices of $V \cup W$.  If $\ep = 1$, we embed the remaining vertices $v_1$ and $w_1$ on the circle $X$ (the axis for $g$), so that $\vf(v_1) = w_1$ (note that $\vf$ preserves $X$ setwise).  If $\ep = 2$, we embed the remaining four vertices $v_1, v_2, w_1, w_2$ on $X$ so that $v$'s and $w$'s alternate around the circle, and $\vf(v_i) = w_i$.

Now we must check the conditions of the Edge Embedding Lemma. The only vertices fixed by any element of $G_1$ are the $2\ep$ vertices embedded on $X$, which are fixed by every $g^i$.  Since all of these motions fix the circle $X$, condition (1) is satisfied.  In every case, for each fixed pair $\{v,w\}$ we can find an arc $A_{vw}$ on $X$ which is disjoint from the vertices and the other arcs (when $\ep = 2$, this is because the $v$'s and $w$'s alternate), so condition (2) is satisfied.  Each $g^i$ fixes $A_{vw}$ pointwise, and $\vf g^i$ fixes $A_{v_iw_i}$ setwise for $i=1,2$ (but does not fix the endpoints of the other arcs, when $\ep = 2$).  So condition (3) is satisfied.  All $\varphi g^{i}$ interchange pairs, but their fixed point sets do not contain any vertices, and no $g^{i}$ interchange pairs, and thus condition (4) is satisfied. Also, $\fix(\varphi g^{i}) \cong S^{1}$, $\fix(\varphi g^{i}) \neq \fix(\varphi g^{k})$ if $i \neq k$, and $\fix(\varphi g^{i}) \neq \fix(g^{k})$, so condition (5) is satisfied. Thus we are able to embed the edges of $K_{n,n}$ in $S^{3}$ such that the resulting embedding is setwise invariant under $G_1$ and we get an embedding $\Gamma$ such that $D_{m}\subseteq \TSG(\Gamma)$.

Now we will apply the Subgroup Lemma to show that we can modify the embedding so that $\TSG = D_m$. We will first assume that $m \ge 3$, and then cover the case when $m=2$. Since $n \geq 3$, we have embedded at least one orbit of size $2m$, containing vertices $v_{0}, ... , v_{m-1}$ and $w_{0}, ... ,w_{m-1}$. Since $g(V) = V$, label the vertices so that $g^i(v_0) = v_i$ and $g^i(w_0)=w_i$. We further define our embedding so that $\varphi(v_{0}) = w_{0}$. Thus $\varphi g^{i} (v_{0}) = g^{-i}(w_0) = w_{m-i}$ and $\varphi g^{i} (w_{0}) = g^{-i}(v_0) = v_{m-i}$. Consider an edge $e_{1} = \overline{v_{0}w_{0}}$.  Then we have that $g^{i}(e_{1}) = g^{i}(\overline{v_{0}w_{0}}) = \overline{v_{i}w_{i}}$ and $\varphi g^{i} (e_{1}) = \varphi g^{i}(\overline{v_{0}w_{0}}) = \overline{w_{m-i}v_{m-i}}$. Thus $\langle e_{1} \rangle _{G_1} = \{ \overline{v_{i}w_{i}}: 0 \leq i \leq m-1\}$. Also consider the edges $e_{2} = \overline{v_{0}w_{1}}$ and $e_{3} = \overline{v_{0}w_{2}}$. We have that $g^{i}(e_{2}) = g^{i}(\overline{v_{0}w_{1}}) = \overline{v_{i}w_{i+1}}$ and $\varphi g^{i} (e_{2}) = \varphi g^{i}(\overline{v_{0}w_{1}}) = \overline{w_{m-i}v_{m-i-1}}$. Thus $ \langle e_{2}\rangle _{G_1} = \{ \overline{v_{i}w_{i+1}}: 0 \leq i \leq m-1\}$. Similarly, $\langle e_{3}\rangle _{G_1} = \{ \overline{v_{i}w_{i+2}}: 0 \leq i \leq m-1\}$.  Thus we have that $\langle e_{1} \rangle_{G_1}$, $\langle e_{2} \rangle_{G_1}$, and $\langle e_{3} \rangle _{G_1}$ are all distinct. Suppose $\psi$ is an automorphism of $K_{n,n}$ which fixes $e_{1}$ pointwise, and fixes $\langle e_{1} \rangle_{G_1}$, $\langle e_{2} \rangle_{G_1}$ and $\langle e_{3} \rangle_{G_1}$ setwise. This implies that $v_{0}$ and $w_{0}$ are both fixed. Since $\langle e_{2} \rangle_{G_1} = \{ \overline{v_{i}w_{i+1}}: 0 \leq i \leq m-1\}$, $e_{2}$ is the only edge in the orbit of $e_{2}$ that is adjacent to $v_{0}$. Since $v_{0}$ is fixed, this means that $e_{2}$ is fixed. Thus $w_{1}$ is also fixed. Similarly, since $\langle e_{3}\rangle _{G_1} = \{ \overline{v_{i}w_{i+2}}: 0 \leq i \leq m-1\}$, $e_{3}$ is the only edge in the orbit of $e_{3}$ that is adjacent to $v_{0}$, and so $e_{3}$ is fixed. This means that $w_{2}$ is fixed. Thus we have that a fixed subgraph isomorphic to $K_{3,1}$. Since $K_{3,1}$ cannot be embedded in $S^{1}$, the Subgroup Lemma implies there exists an embedding $\Gamma'$ of $K_{n,n}$ such that $\TSG(\Gamma') = D_{m}$, when $m \ge 3$.  Moreover, $e_1$ is not fixed by any non-trivial element of $\TSG(\Gamma')$, so by the Subgroup Corollary there is another embedding $\Gamma''$ of $K_{n,n}$ such that $\TSG(\Gamma'') = Z_m$.

We still need to deal with the case when $m = 2$. We will divide this into two subcases, when $n$ is even and when $n$ is odd.  If $n = 2r$, then $2n = 4r$, and all of the vertices of $K_{n,n}$ are embedded in $r$ orbits, each containing $4$ vertices. Consider one of these $r$ orbits, containing vertices $v_{0}, w_{0}, v_{1},$ and $w_{1}$. As when $m \geq 3$, choose the labels so that $g(v_0) = v_1$, $g(w_0) = w_1$ and $\vf(v_0) = w_0$.  Consider edges $e_{1} = \overline{v_{0}w_{0}}$, and $e_{2} = \overline{v_{0}w_{1}}$.  Then $\langle e_{1} \rangle _{G_1} = \{ \overline{v_{0}w_{0}}, \overline{v_{1}w_{1}}\}$ and $\langle e_{2}\rangle _{G_1} = \{ \overline{v_{0}w_{1}}, \overline{v_{1}w_{0}}\}$.  Thus the orbits of $e_{1}$ and $e_{2}$ are distinct. Suppose $\psi$ is an automorphism of $K_{n,n}$ which fixes $e_{1}$ pointwise, and fixes $\langle e_{1} \rangle_{G_1}$ and $\langle e_{2} \rangle_{G_1}$ setwise. This implies that $v_{0}$ and $w_{0}$ are both fixed. Since $e_{2}$ is the only edge in $\langle e_{2}\rangle _{G_1}$ that is adjacent to $v_{0}$, this means that $e_{2}$ is fixed, which implies that $w_{1}$ is also fixed. Thus we have that $v_{0}, w_{0},$ and $w_{1}$ are all fixed. Since $v_{1}$ is the only other vertex in the orbit it must be that $v_{1}$ is also fixed. Thus we have a subgraph $K_{2,2}$ fixed pointwise by $\psi$. Since $n \geq 3$ we must have another orbit of $2m = 4$ vertices, which we will call $v_{0}', w_{0}', v_{1}',$ and $w_{1}'$, labeled using the same conventions used for the first orbit. Consider the edge $f_{1} = \overline{w_{0}v_{0}'}$. Then $\langle f_{1} \rangle_{G_1} = \{ \overline{w_{0}v_{0}'}, \overline{v_{0}w_{0}'}, \overline{v_{1}w_{1}'}, \overline{w_{1}v_{1}'} \}$. Assume $\psi$ also fixes $\langle f_1 \rangle_{G_1}$ setwise.  Since $f_{1}$ is the only edge in $\langle f_{1} \rangle _{G_1}$ that is adjacent to $w_{0}$, then $f_{1}$ is fixed, and thus $v_{0}'$ is fixed. So we now have a subgraph $K_{3,2}$ that is fixed pointwise by $\psi$. Since $K_{3,2}$ cannot be embedded in $S^{1}$, by the Subgroup Lemma there is an embedding $\Gamma'$ of $K_{n,n}$ so that $\TSG( \Gamma') = D_{2}$. 

If $n=2r+1$ then all of the vertices of $K_{n,n}$, except for one from each of $V$ and $W$, are embedded in $r$ orbits, each containing $4$ vertices of $V \cup W$. Since $n \geq 3$, $r \neq 0$, and thus there must be $4$ vertices embedded in one of these orbits. As above, we can fix $e_{1} = \overline{v_{0}w_{0}}$ and this will cause all four of these vertices to be fixed. Now consider the vertices $v_{2}$ and $w_{2}$ which are the two vertices embedded on the axis of $g$. Consider the edge $f_{1} = \overline{w_{0}v_{2}}$. Since $\langle v_{2} \rangle _{G_1} = \{v_{2}, w_{2}\}$ and $w_{0}$ is not fixed by any element of $G_1$, we have that $f_{1}$ is the only edge in $\langle f_{1} \rangle _{G_1}$ that is adjacent to $w_{0}$. Since $w_{0}$ is fixed, this implies that $v_{2}$ is fixed. Thus we have that a subgraph $K_{3,2}$ is fixed pointwise. Since $K_{3,2}$ cannot be embedded in $S^{1}$ there is an embedding $\Gamma'$ of $K_{n,n}$ so that $\TSG( \Gamma') = D_{2}$.

In either case, when $m = 2$, the edge $e_1 = \overline{v_0w_0}$ is not fixed by any non-trivial element of $\TSG(\Gamma')$.  So by the Subgroup Corollary there is another embedding $\Gamma''$ of $K_{n,n}$ such that $\TSG(\Gamma'') = \Z_2$.
\end{proof} \medskip

\begin{lemma} \label{L:n=m/2}
If $n \geq 3$, $m$ is even, $n\equiv 0 \pmod{\frac{m}{2}}$ and $H = \Z_m$ or $D_m$, then there exists an embedding, $\Gamma$, of $K_{n,n}$ in $S^{3}$ such that $\TSG(\Gamma)=H$.
\end{lemma}

\begin{proof}
We have already dealt with the case when $n \equiv 0 \pmod{m}$ in Lemma \ref{L:n=0,1,2}, so we may assume $n = mr+\frac{m}{2}$ for some $r \in \mathbb{Z}$. We will use the group of motions $G_{2}$. Observe that for all $i$, where $i \neq 0$ and $i \neq \frac{m}{2}$, $h^{i}$ has no fixed points. Also notice that $h^{m/2}$ is a rotation of order 2 with axis $Y$. Since $n = mr+\frac{m}{2}$, then $2n = 2mr +m$. Pick a small ball, $M$, such that for each non-trivial $h \in G_2$, $h(M) \cap M = \emptyset$ (i.e. $G_2$ acts freely on $M$).  Pick $r$ points $\{p_1, \dots, p_r\}$ in $M$; then each of these points has an orbit containing $2m$ points. Embed $2mr$ of the vertices of $V \cup W$ in $S^3$ so that vertices of $V$ are embedded as the points $h^{2i}(p_j)$ and $\varphi h^{2i}(p_j)$, and vertices of $W$ are embedded as the points $h^{2i+1}(p_j)$ and $\varphi h^{2i+1}(p_j)$. Then we have that $h(V) = W$ and $\varphi(V) = V$. Now we must embed the remaining $m$ vertices of $V \cup W$. Take a point $q$ on $Z - (X \cup Y)$. Then the orbit of $q$ under the action of $G_2$ contains $m$ points. Embed a vertex of $V$ as $q$, and embed the rest of the $m$ vertices of $V \cup W$ as the points $h^{i}(q)$, alternating $v$'s and $w$'s. Observe that $h^{i}(q)$ and $h^{i+\frac{m}{2}}(q)$ are both on the axis of $\vf h^{-2i}$. Now we must check that we can embed the edges of the graph. 

If $4 \vert m$, then $\frac{m}{2}$ is even, so $i$ and $(i+\frac{m}{2})$ have the same parity. Thus $h^{i}(q)$ and $h^{i+\frac{m}{2}}(q)$ are both in $V$ or both in $W$. So no adjacent pairs lie on the same axis, and hence no such pairs are fixed by any elements of $G_2$. So conditions (1), (2), and (3) of the Edge Embedding Lemma are satisfied. Also, since $\frac{m}{2}$ is even, $h^{m/2}(V) = V$, so $h^{m/2}$ does not interchange pairs bounding an edge. The only other rotations are $\vf h^i$, which also send $V$ to $V$, so no pairs are interchanged. Thus conditions (4) and (5) are also satisfied. 

Now we will consider the case when $4 \nmid m$, so $h^{m/2}(V) = W$. None of the $2mr$ vertices embedded in $S^3 - (X \cup Y \cup Z)$ are fixed by any elements, and each of the pairs in the $m$ vertices is fixed by only one element, and thus condition (1) of the Edge Embedding Lemma is satisfied. Since $h$ is a glide rotation, each of the pairs $\{h^i(q), h^{i+\frac{m}{2}}(q)\}$ bounds an arc on the axis of $\vf h^{-2i}$ that is disjoint from the other vertices and the other such arcs. Thus condition (2) is satisfied. Each $\varphi h^{-2i}$ fixes an adjacent pair of vertices, but for each of those pairs the arc between the vertices lies on the axis of involution. Thus the arc is fixed by the same involution that the pair of vertices is fixed by. Also, $h^{m/2}$ interchanges the end points of each arc in the $m$ vertices, but it fixes each arc setwise. So condition (3) holds. Only $h^{m/2}$ and $\varphi h^{\frac{m}{2}-2i}$ interchange pairs of adjacent vertices. But $h^{m/2}$ fixes no vertices and, since $\frac{m}{2}-2i$ is odd, neither does $\varphi h^{\frac{m}{2}-2i}$, and so condition (4) is met. Both $h^{m/2}$ and $\varphi h^{\frac{m}{2}-2i}$ have fixed point sets that are homeomophic to $S^{1}$ and also different than all fixed point sets of other elements in $G_2$. Thus condition (5) is satisfied. Therefore we are able to embed the edges of $K_{n,n}$ in $S^{3}$ such that the resulting embedding is setwise invariant under $G_2$ and we get an embedding $\Gamma$ such that $D_{m}\subseteq\TSG(\Gamma)$.

Now we will apply the Subgroup Lemma to show that we can modify the embedding so that $\TSG = D_m$. Recall that $m$ must be even and $n = mr+\frac{m}{2}$, so $2n = 2mr + m$. If $m = 2$, then $n \equiv 1 \pmod{m}$, which we have already dealt with in Lemma \ref{L:n=0,1,2}.  So $m \geq 4$.  We first suppose $m \geq 6$. Then there an $m$-cycle $(p_0p_1\dots p_{m-1})$, where $p_{2i}$ is in $V$ and $p_{2i+1}$ is in $W$. So $h^{j}(p_{i}) = p_{i+j}$ and $\varphi (p_{i}) = p_{m-i}$, and thus $\varphi h^{j} (p_{i}) = \varphi (p_{i+j}) = p_{m-i-j}$. Consider the edge $e_{1} = \overline{p_{0}p_{1}}$. We have that $h^{j}(e_{1}) =\overline{p_{j}p_{j+1}}$ and $\varphi h^{j} (e_{1}) = \overline{p_{m-j}p_{m-j-1}}$. Thus $\langle e_{1}\rangle _{G_2} = \{ \overline{p_{j}p_{j+1}}: 0 \leq j \leq m-1\}$. Suppose $\psi$ is an automorphism of $K_{n,n}$ which fixes $e_{1}$ pointwise, and fixes $\langle e_{1} \rangle_{G_2}$ setwise. This means that $p_{0}$ and $p_{1}$ are both fixed. The only edges in $\langle e_{1}\rangle _{G_2}$ adjacent to $p_{1}$ are $e_{1} = \overline{p_{0}p_{1}}$ and $e_{2} = g(e_{1}) = \overline{p_{1}p_{2}}$. Since $e_{1}$ and $p_{1}$ are fixed, so is $e_{2}$. Thus $p_{2}$ is also fixed. Similarly, the only edges in $\langle e_{1}\rangle _{G_2}$ adjacent to $p_{2}$ are $e_{2} = \overline{p_{1}p_{2}}$ and $e_{3} = g(e_{2}) = \overline{p_{2}p_{3}}$. Since $e_{2}$ and $p_{2}$ are fixed, so is $e_{3}$. Thus $p_{3}$ is also fixed. Continuing inductively, we can fix the whole $m$-cycle. Since $m \geq 6$ we have fixed a subgraph isomorphic to $K_{3,3}$ that cannot be embedded in $S^{1}$. Thus if $m \geq 6 $ there exists an embedding $\Gamma'$ so that $\TSG(\Gamma') = D_{m}$. Moreover, since $e_1$ is only fixed pointwise by the identity, the Subgroup Corollary implies there is another embedding $\Gamma''$ such that $\TSG(\Gamma'') = \Z_m$.

We still need to consider when $m = 4$. Since $m = 4$ and $n \geq 3$, there must be at least one orbit of size $2m=8$. If this is the orbit of point $p$, we embed the vertices so that $\{p, h^2(p), \vf(p), \vf h^2(p)\}$ are vertices in $V$ and $\{h(p), h^3(p), \vf h(p), \vf h^3(p)\}$ are vertices in $W$.  So $(p\ h(p)\ h^2(p)\ h^3(p))$ is a 4-cycle which alternates between vertices in $V$ and $W$; using the same argument as above, any automorphism $\psi$ of $K_{n,n}$ which fixes $e_1 = \overline{p\, h(p)}$ pointwise and $\langle e_{1} \rangle_{G_2}$ setwise will fix the entire 4-cycle pointwise, and hence fix a subgraph isomorphic to $K_{2,2}$.  Now consider $e_2=\overline{p\, \vf h(p)}$, and suppose that $\psi$ also fixes $\langle e_{2}\rangle_{G_2}$ setwise. Notice that $\langle e_{2}\rangle_{G_2} =$ $ \{ \overline{p\, \vf h(p)}, $ $\overline{h(p)\, \vf(p)}, $ $\overline{h^2(p)\, \vf h^3(p)}, $ $\overline{h^3(p)\, \vf h^2(p)} \}$. Thus $e_{2}$ is the only edge in $\langle e_{2}\rangle _{G_2}$ that is adjacent to $p$. Since $p$ is fixed this implies that $e_{2}$ is fixed, and thus $\vf h(p)$ is also fixed. Thus we have fixed a subgraph $K_{3,2}$, induced by the vertices $\{p, h^2(p)\}$ in $V$ and $\{h(p), h^3(p), \vf h(p)\}$ in $W$. Since $K_{3,2}$ cannot be embedded in $S^{1}$ there exists an embedding $\Gamma'$ so that $\TSG(\Gamma') = D_{4}$.  As above, since $e_1$ is only fixed pointwise by the identity, the Subgroup Corollary implies there is another embedding $\Gamma''$ such that $\TSG(\Gamma'') = \Z_4$.
\end{proof}
\medskip

\begin{lemma} \label{L:n=m/2+2}
If $n \geq 3$, $4\vert m$,  $n\equiv 2 \pmod{\frac{m}{2}}$ and $H = \Z_m$ or $D_m$, then there exists an embedding, $\Gamma$, of $K_{n,n}$ in $S^{3}$ such that $\TSG(\Gamma)=H$.
\end{lemma}

\begin{proof}
If $n \equiv 2 \pmod{m}$, then we are done by Lemma \ref{L:n=0,1,2}, so we may assume that $n = mr+\frac{m}{2}+2$ for some integer $r$, and hence $2n = 2mr +m+4$. We will use the group of motions $G_{3}$. Pick a small ball, $M$, such that for each non-trivial $h \in G_3$, $h(M) \cap M = \emptyset$ (i.e. $G_3$ acts freely on $M$).  Pick $r$ points $\{p_1, \dots, p_r\}$ in $M$; then each of these points has an orbit containing $2m$ points. Embed $2mr$ of the vertices of $V \cup W$ so that vertices of $V$ are embedded as the points $j^{2i}(p_k)$ and $\varphi j^{2i}(p_k)$, and vertices of $W$ are embedded as the points $j^{2i+1}(p_k)$ and $\varphi j^{2i+1}(p_k)$. Then $j(V) = W$ and $\varphi(V) = V$. Now we must embed the next $m$ vertices of $V \cup W$. Take a point $q$ on $Z - (X \cup Y)$. Then the orbit of $q$ under the action of $G_3$ contains $m$ points. Embed a vertex of $V$ as $q$, and embed the rest of the $m$ vertices as the points $j^{i}(q)$, alternating $v$'s and $w$'s. Observe that $j^{i}(q)$ and $j^{i+\frac{m}{2}}(q)$ are both on the axis for $\vf j^{-2i}$. Since $\frac{m}{2}$ is even the two vertices embedded on the axis of $\vf j^{-2i}$ are either both in $V$ or both in $W$. Lastly, we must embed the last four vertices of $V \cup W$.  Let $y$ be a point of intersection of the circle $Y$ and the circle $Z$ (the axis of $\vf$).  Embed the last four vertices as $y, j(y), j^2(y), j^3(y)$, placing vertices of $V$ at $y$ and $j^2(y)$ and vertices of $W$ at $j(y)$ and $j^3(y)$.  Since the action of $j$ on $Y$ has order 4, $j^2(y)$ is the other point of intersection of $Y$ and $Z$, and $\vf j(y) = j^3(y)$.

Now we must check that we can embed the edges of the graph. The only elements which fix a pair of adjacent vertices are $j^{4i}$ (which fix the last four vertices embedded), which all have the same fixed point set $Y$, and thus condition (1) of the Edge Embedding Lemma is satisfied. The pairs of adjacent vertices fixed by $j^{4i}$ all bound arcs on $Y$ that are disjoint from other vertices and arcs (since vertices of $V$ and $W$ alternate around $Y$), and thus condition (2) is satisfied. The elements $j^{4i}$ fix the circle $Y$ pointwise, and thus condition (3) is satisfied. The only elements of $G_{3}$ which interchange adjacent pairs are $\varphi j^{2i+1}$, which interchange the two pairs embedded on $Y$. $2mr$ vertices of $V\cup W$ are embedding in the orbit of $M$, which is disjoint from the axes of all the elements of $G_3$, another $m$ vertices are embedded as $j^{i}(q)$, which is on the axis for $\varphi j^{-2i}$, and the last 4 vertices (on $Y$) are embedded on the axes for $\vf$ or $\vf j^2$.  So no vertices are embedded on the axes for $\vf j^{2i+1}$.  Thus these elements of $G_3$ fix no vertices and condition (4) is satisfied. Again, only a $\varphi j^{i}$ can possibly interchange a pair of vertices, and each $\varphi j^{i}$ has a fixed point set that is homeomorphic to $S^{1}$ and is unique, so condition (5) is satisfied. Thus we are able to embed the edges of $K_{n,n}$ in $S^{3}$ such that the resulting embedding is setwise invariant under $G_3$.

Now we will apply the Subgroup Lemma to show that we can modify the embedding so that $\TSG = D_m$. If $m = 4$, then $n = 4r + 2 + 2 = 4(r+1)$, so $n \equiv 0 \pmod{m}$.  This case has already been dealt with in Lemma \ref{L:n=0,1,2}, so we may assume $m \geq 8$.  We will have an $m$-cycle under the action of $G_3$, and exactly as in Lemma \ref{L:n=m/2} we can show that if an automorphism of $K_{n,n}$ fixes an edge of this $m$-cycle pointwise, and the orbit of that edge setwise, then it must fix the entire cycle, which contains a subgraph isomorphic to $K_{4,4}$.  Since $K_{4,4}$ cannot be embedded in $S^{1}$, there exists an embedding $\Gamma'$ so that $\TSG(\Gamma') = D_{m}$.  Moreover, also as in Lemma \ref{L:n=m/2}, there is an edge which is not fixed by any nontrivial element of the topological symmetry group, so the Subgroup Corollary implies there is another embedding $\Gamma''$ such that $\TSG(\Gamma'') = \Z_m$.
\end{proof}
\medskip

Combining Lemmas \ref{L:cyclic1}, \ref{L:n=0,1,2}, \ref{L:n=m/2} and \ref{L:n=m/2+2} gives us Theorem \ref{T:cyclic}.

\begin{them1}
Let $n>2$ and let $K_{n,n}$ be the complete bipartite graph on $n,n$ vertices. Then there exists an embedding, $\Gamma$, of $K_{n,n}$ in $S^{3}$ such that $\TSG(\Gamma)=H$ for $H=\mathbb{Z}_{m}$ or  $D_{m}$ if and only if one of the following conditions hold:
	\begin{enumerate}
	\item $n\equiv0,1,2 \pmod{m}$,
	\item $n\equiv0 \pmod{\frac{m}{2}}$ if $m$ is even,
	\item $n\equiv2 \pmod{\frac{m}{2}}$ if $m$ is even and $4|m$.
\end{enumerate}
\end{them1}

\section{Necessity of conditions for $\TSG(\Gamma) = \mathbb{Z}_{r}\times\mathbb{Z}_{s}$ or $(\Z_r \x \Z_s) \ltimes \Z_2$} \label{S:necessity}

In this section we will prove the necessity of the conditions in Theorem \ref{T:product}. Recall that \begin{itemize}
	\item $\Z_r \times \Z_s = \langle g, h \vert g^r = h^s = 1, gh = hg \rangle$
	\item $(\Z_r \times \Z_s) \ltimes \Z_2 = \langle g, h, \varphi \vert g^r = h^s = \varphi^2 = 1, gh = hg, \varphi g = g^{-1} \varphi, \varphi h = h^{-1} \varphi \rangle$
\end{itemize}

Since $\mathbb{Z}_{r}$ and $\mathbb{Z}_{s}$ are both subgroups of $\mathbb{Z}_{r}\times\mathbb{Z}_{s}$ then the conditions of Lemma \ref{L:cyclic1} must hold for both $r$ and $s$.  We are assuming that $r \vert s$, so the conditions for $s$ are strictly stronger.  So if $K_{n,n}$ has an embedding with topological symmetry group $\mathbb{Z}_{r}\times\mathbb{Z}_{s}$ or $(\Z_r \x \Z_s) \ltimes \Z_2$, then one of the following holds:
\begin{enumerate}
\item $n \equiv 0,1,2 \pmod {s}$
\item $n \equiv 0 \pmod {\frac{s}{2}}$ for $s$ even
\item $n \equiv 2 \pmod {\frac{s}{2}}$ for $s$ even and $4\vert s$
\end{enumerate}
 
However, these conditions are not all sufficient.  In Lemmas \ref{L:1mods} through \ref{L:s/2+2mods}, we will prove that in some of these cases it is {\em not} possible to embed $K_{n,n}$ such that its topological symmetry group contains $\mathbb{Z}_{r}\times\mathbb{Z}_{s}$.  Since $\mathbb{Z}_{1}\times\mathbb{Z}_{s}=\mathbb{Z}_{s}$ and $\mathbb{Z}_{2}\times\mathbb{Z}_{2}=D_{2}$ and we have already considered cyclic and dihedral groups, we shall assume $r\geq 2$ and $s \geq 3$.

\begin{lemma} \label{L:1mods}
If $n\equiv 1 \pmod {s}$ there does not exist an embedding $\Gamma$ of $K_{n,n}$ such that $\mathbb{Z}_{r}\times\mathbb{Z}_{s}\subseteq \TSG(\Gamma)$.
\end{lemma}
\begin{proof}
Assume there is such an embedding $\Gamma$. Let $\alpha$ and $\beta$ be automorphisms in $\TSG(\Gamma)$ such that $\langle\alpha,\beta \rangle\cong\mathbb{Z}_{r}\times\mathbb{Z}_{s}$. Suppose $\b$ interchanges the vertex sets $V$ and $W$. By the Automorphism Theorem, this means $V\cup W$ is either partitioned into $s$-cycles of $\b$, or $s$-cycles and a single 4-cycle.  But since $n \equiv 1 \pmod{s}$, $2n \equiv 2 \pmod{s}$ (with $s \geq 3$), so this is not possible.  Hence $\b$ fixes $V$ and $W$ setwise.

Consider the case when $\alpha$ also fixes $V$ setwise. By the Automorphism Theorem, $\alpha$ must fix one vertex, $v$, of $V$ and $V-\fix(\alpha)$ is partitioned into $r$-cycles. Then by the Orbits Lemma, $\beta(v)=\{v\}$ since $\{v\}$ is the only $\alpha$-orbit of length one in $V$. So $v$ is a fixed vertex of $\beta$. This contradicts the Disjoint Fixed Points Lemma. So $\a$ must interchange $V$ and $W$.

By the Automorphism Theorem we know that $\beta$ must fix one vertex of each of $V$ and $W$. Call these vertices $v,w$. By the Orbits Lemma, $\alpha(v)=w$ and $\alpha(w)=v$, so $\alpha$ fixes the embedded edge $\overline{vw}$ in $\Gamma$ setwise, and hence fixes some point in its interior. Since $\beta$ fixes the edge pointwise, $\fix(\alpha)\cap\fix(\beta) \ne \emptyset$. This contradicts the Disjoint Fixed Points Lemma. Thus if $n\equiv 1 \pmod{s}$, there does not exist an embedding, $\Gamma$, of $K_{n,n}$ such that $\TSG(\Gamma) = \mathbb{Z}_{r}\times\mathbb{Z}_{s}$.
\end{proof}\\

\begin{lemma} \label{L:2mods1}
If $n\equiv 2 \pmod {s}$ and $s$ is odd, there does not exist an embedding $\Gamma$ of $K_{n,n}$ such that $\mathbb{Z}_{r}\times\mathbb{Z}_{s} \subseteq \TSG(\Gamma)$.
\end{lemma}
\begin{proof}
Assume there is such an embedding $\Gamma$. Let $\alpha$ and $\beta$ be automorphisms in $\TSG(\Gamma)$ such that $\langle\alpha,\beta \rangle\cong\mathbb{Z}_{r}\times\mathbb{Z}_{s}$. Since $r\vert s$, $r$ is also odd. Since their orders are odd, $\alpha$ and $\beta$ both fix $V$ setwise. So by the Automorphism Theorem we know that  $\alpha$ either fixes two vertices of $V$ or has a 2-cycle in $V$. But $\a$ cannot have any 2-cycles, since it has odd order.  Therefore $\alpha$ fixes two vertices, $v_{1}$ and $v_{2}$, and $V-\fix(\a)$ is partitioned into $r$-cycles. By the Orbits Lemma, $\beta(\{v_{1},v_{2}\})=\{v_{1},v_{2}\}$. But, by the Disjoint Fixed Points Lemma, $\beta$ cannot fix either $v_{1}$ or $v_{2}$. So $\beta$ has a cycle of length 2. This is impossible since $s$ is odd. Thus there does not exist an embedding $\Gamma$ of $K_{n,n}$ such that $\mathbb{Z}_{r}\times\mathbb{Z}_{s} \subseteq \TSG(\Gamma)$.
\end{proof}\\

\begin{lemma} \label{L:2mods2}
If $n\equiv 2\pmod{s}$ and there exists an embedding $\Gamma$ of $K_{n,n}$ such that $\mathbb{Z}_{r}\times\mathbb{Z}_{s} \subseteq \TSG(\Gamma)$, then $r=2$ or $r=4$.
\end{lemma}
\begin{proof}
Assume there is such an embedding $\Gamma$, and $r \neq 2, 4$.  By Lemma \ref{L:2mods1}, $s$ must be even.  Let $\alpha$ and $\beta$ be automorphisms in $\TSG(\Gamma)$ such that  $\langle\alpha,\beta \rangle = \mathbb{Z}_{r}\times\mathbb{Z}_{s}$. We consider four cases, depending on whether $\a$ and $\b$ fix or interchange the vertex sets $V$ and $W$.  Observe there is an integer $k$ such that $n = ks + 2$, and $2n = 2ks + 4$.  Also observe that, if $a < r$ and $b < s$, then $\langle \a^a, \b^b\rangle$ is a product of cyclic groups which is not itself cyclic.

We first suppose that $\a(V) = V$ and $\b(V) = V$.  By the Automorphism Theorem, $\b$ either fixes two vertices in each of $V$ and $W$, or has a 2-cycle in each of $V$ and $W$.  So $\b^2$ fixes vertices $v_1, v_2$ in $V$ and $w_1, w_2$ in $W$. By the Orbits Lemma, $\a(\{v_1, v_2\}) = \{v_1, v_2\}$, so $\a^2$ must also fix $v_1$ and $v_2$.  But, since $r \neq 2$, $\langle \a^2, \b^2\rangle$ is a product of cyclic groups, so this contradicts the Disjoint Fixed Points Lemma.

Next we suppose that $\a(V) = V$ and $\b(V) = W$.  Since $s$ is even, $r \neq 2, 4$ and $r \vert s$, $s > 4$.  So by the Automorphism Theorem $\b$ partitions $V \cup W$ into $s$-cycles, along with a single 4-cycle.  Hence $\b^4$ fixes vertices $v_1, v_2$ in $V$ and $w_1, w_2$ in $W$. As in the last case, $\a^2$ must also fix $v_1$ and $v_2$, which contradicts the Disjoint Fixed Points Lemma.

Now suppose that $\a(V) = W$ and $\b(V) = V$.  Since $r \neq 2$ or $4$, $\a$ partitions $V \cup W$ into $r$-cycles, along with a single 4-cycle.  Now we simply repeat the argument in the previous paragraph, reversing the roles of $\a$ and $\b$.

Finally, suppose that $\a(V) = W$ and $\b(V) = W$.  As before, $\b^4$ fixes vertices $v_1, v_2$ in $V$ and $w_1, w_2$ in $W$.  Now $\a^2$ fixes $V$ setwise, so, as in the previous cases, $\a^4$ must fix the points $v_1$ and $v_2$.  Since $r \neq 2$ or $4$, and $r$ must be even, this contradicts the Disjoint Fixed Points Lemma.

Therefore, if there is such an embedding $\Gamma$, then $r = 2$ or $4$.
\end{proof}\\

\begin{lemma} \label{L:s+2mod2sr=2}
If $n\equiv s+2 \pmod{2s}$, $r=2$ and $\frac{s}{2}$ is odd, then there does not exist an embedding $\Gamma$ of $K_{n,n}$ such that $\mathbb{Z}_{r}\times\mathbb{Z}_{s} \subseteq \TSG(\Gamma)$.
\end{lemma}
\begin{proof}
Assume there is such an embedding $\Gamma$. Let $\alpha$ and $\beta$ be diffeomorphisms of $(S^3, \Gamma)$ such that  $H = \langle\alpha,\beta \rangle = \mathbb{Z}_{2}\times\mathbb{Z}_{s}$.  By Corollary \ref{C:commute}, the motions in $H$ are all combinations of rotations around a pair of complementary geodesic circles $X$ and $Y$.  Any point in $S^3 - (X \cup Y)$ has an orbit of size $2s$ under the action of $H$.

Observe that $\a$, $\b^{s/2}$ and $\a\b^{s/2}$ are three distinct elements of order 2 in $H$, all fixing $X$ and $Y$ setwise.  One of them must be rotation around $X$, another rotation around $Y$, and the third the central inversion, since these are the only combinations of rotations around $X$ and $Y$ with order 2.  Since $\langle \a, \b\rangle = \langle \a, \a\b \rangle = \langle \a\b^{s/2}, \b\rangle = H$, we may assume without loss of generality that $\a$ is rotation around $X$ and $\b^{s/2}$ is rotation about $Y$.

Since $n = 2ks + s + 2$, for some integer $k$, we know $2n = 4ks + 2s + 4$.  If $\b(V) = W$, then the Automorphism Theorem tells us that $\b$ must have a 4-cycle in $V \cup W$.  But this means that $4\vert s$, contradicting the assumption that $\frac{s}{2}$ is odd.  Hence $\b(V) = V$.  By the Automorphism Theorem, $\b$ either fixes two vertices in each of $V$ and $W$, or has a two-cycle in each of $V$ and $W$.  In either case, $\b^2$ fixes two vertices in each of $V$ and $W$.

Suppose we also have $\a(V) = V$.  Then the orbit of a vertex in $S^3 - (X \cup Y)$ consists entirely of vertices in the same vertex set, and so we must have at least $s+2$ vertices from each vertex set embedded on $X \cup Y$. Any motion which fixes more than 4 vertices of $V \cup W$, including vertices from both $V$ and $W$, fixes a subgraph of $K_{n,n}$ which cannot be embedded in a circle.  Since there are at least $s+2 \geq 6$ vertices in each of $V$ and $W$, Smith Theory implies that we cannot embed vertices of both $V$ and $W$ on circles $X$ and $Y$.  Hence we must embed all $s+2$ vertices of $V$ on $X$ and all $s+2$ vertices of $W$ on $Y$.  Since $\b^2$ fixes two vertices in each of $V$ and $W$, it fixes vertices on both $X$ and $Y$, and therefore fixes $X\cup Y$.  But this would imply, by Smith Theory, that $\b^2$ is the identity, which is not the case.

Now suppose that $\a(V) = W$.  Since $\a$ is rotation about $X$, we cannot embed any vertices on $X$.  So we must embed $4ks+2s$ vertices in $S^3-(X\cup Y)$ and the remaining four vertices on $Y$ (two vertices from each of $V$ and $W$).  These vertices are fixed by $\b^{s/2}$, so the edges between them must also be in the fixed point set of $\b^{s/2}$, which is $Y$.  So the vertices from $V$ and $W$ must alternate around $Y$. Denote these vertices $v_1, v_2, w_1, w_2$, labeled so that $\a(v_i) = w_i$.  But then $\overline{v_1w_1}$ and $\overline{v_2w_2}$ must be embedded as semicircles of $Y$.  Since the endpoints are different, these semicircles must intersect, which contradicts our assumption that $\Gamma$ is an embedding.  

So the embedding $\Gamma$ cannot exist.
\end{proof}\\

\begin{lemma} \label{L:s+2mod2sr=4}
If $n\equiv s+2 \pmod{2s}$ and $r=4$, then there does not exist an embedding $\Gamma$ of $K_{n,n}$ such that $\mathbb{Z}_{r}\times\mathbb{Z}_{s} \subseteq \TSG(\Gamma)$.
\end{lemma}
\begin{proof}
Assume there is such an embedding $\Gamma$. Let $\alpha$ and $\beta$ be diffeomorphisms of $(S^3, \Gamma)$ such that  $H = \langle\alpha,\beta \rangle = \mathbb{Z}_{4}\times\mathbb{Z}_{s}$.  By Corollary \ref{C:commute}, the motions in $H$ are all combinations of rotations around a pair of complementary geodesic circles $X$ and $Y$.  Any point in $S^3 - (X \cup Y)$ has an orbit of size $4s$ under the action of $H$.

Since $n = 2ks + s + 2$, for some integer $k$, we know $2n = 4ks + 2s + 4$.  So at least $2s+4$ vertices of $K_{n,n}$ must be embedded on $X \cup Y$. Observe that $\a^2$, $\b^{s/2}$ and $\a^2\b^{s/2}$ are three distinct elements of order 2 in $H$, all fixing $X$ and $Y$ setwise.  One of them must be rotation around $X$, another rotation around $Y$, and the third the central inversion, since these are the only combinations of rotations around $X$ and $Y$ with order 2.  Any motion which fixes more than 4 vertices of $V \cup W$, including vertices from both $V$ and $W$, fixes a subgraph of $K_{n,n}$ which cannot be embedded in a circle.  Since there are at least $s+2 \geq 6$ vertices in each of $V$ and $W$, Smith Theory implies that we cannot embed vertices of both $V$ and $W$ on circles $X$ and $Y$.  Hence we must embed all $s+2$ vertices of $V$ on $X$ and all $s+2$ vertices of $W$ on $Y$.

Since $\a$ and $\b$ fix $X$ and $Y$ setwise, they must also fix $V$ and $W$ setwise.  Since $s+2$ is not divisible by $s$, $\b$ must divide each set of $s+2$ vertices into an $s$-cycle and either two fixed points or a 2-cycle.  Then $\b^2$ fixes points in both $V$ and $W$, and therefore must fix both $X$ and $Y$ pointwise.  But this would mean that $\b^2$ is the identity, which is impossible since $s > 2$.  So the embedding $\Gamma$ cannot exist.
\end{proof}\\

\begin{lemma}\label{L:s/2mods}
If $n \equiv \frac{ls}{2}\pmod{rs}$, $s$ is even, $s>4$, and $l$ is odd, then there does not exist an embedding $\Gamma$ of $K_{n,n}$ such that $\mathbb{Z}_{r}\times\mathbb{Z}_{s} \subseteq \TSG(\Gamma)$.
\end{lemma}

\begin{proof}
Assume there is such an embedding $\Gamma$. Let $\alpha$ and $\beta$ be diffeomorphisms of $(S^3, \Gamma)$ such that  $H = \langle\alpha,\beta \rangle = \mathbb{Z}_{r}\times\mathbb{Z}_{s}$.  By Corollary \ref{C:commute}, the motions in $H$ are all combinations of rotations around a pair of complementary geodesic circles $X$ and $Y$.  Any point in $S^3 - (X \cup Y)$ has an orbit of size $rs$ under the action of $H$.  Since $2n = 2krs + ls$ for some integer $k$, at least $ls$ vertices will need to be embedded on $X\cup Y$.  From the Automorphism Theorem, the only way $\b$ can act on $V \cup W$ is if $\b(V) = W$, and the action of $\b$ partitions $V \cup W$ into $s$-cycles.  Hence each of $X$ and $Y$ must contain either vertices of both $V$ and $W$ or no vertices at all.

Suppose that $\a$ is the combination of a rotation of order $a$ about $X$ and order $b$ around $Y$ (with $\lcm(a,b) = r$), and $\b$ is the combination of a rotation of order $c$ around $X$ and order $d$ around $Y$ (with $\lcm(c,d) = s$).  If $c < s$, then $\b^c$ fixes $Y$ pointwise.  On the other hand, if $c = s$, then $a \vert c$ (since $a\vert r$ and $r \vert s$), and $c = ap$ for some $p$.  So $\b^p$ and $\a$ have the same action on $Y$, and $\a \b^{-p}$ fixes $Y$.  So there is some non-trivial element of $H$ which fixes $Y$.  Similarly, there is a non-trivial element fixing $X$ pointwise.  By Smith Theory, this means $X$ and $Y$ can each contain at most 4 vertices, so $ls \leq 8$.   Hence $ls = s = 6$ or $8$.  If $s = 6$, then we must embed 4 vertices on $X$ and 2 vertices on $Y$ (or vice versa).  But then $c$ and $d$ are both 2 or 4, and $\lcm(c,d) \neq s$.  Similarly, if $s = 8$, we must embed 4 vertices on each circle, and $c$ and $d$ are again either 2 or 4, so $\lcm(c,d) \neq s$.  So the embedding $\Gamma$ cannot exist.
\end{proof}\\

\begin{lemma}\label{L:s/2+2mods}
If $n \equiv \frac{s}{2} +2 \pmod{s}$, $4\vert s$ and $s>4$ there does not exist an embedding $\Gamma$ of $K_{n,n}$ such that $\mathbb{Z}_{r}\times\mathbb{Z}_{s} \subseteq \TSG(\Gamma)$.
\end{lemma}
\begin{proof}
Assume there is such an embedding $\Gamma$. Let $\alpha$ and $\beta$ be diffeomorphisms of $(S^3, \Gamma)$ such that  $H = \langle\alpha,\beta \rangle = \mathbb{Z}_{r}\times\mathbb{Z}_{s}$.  By Corollary \ref{C:commute}, the motions in $H$ are all combinations of rotations around a pair of complementary geodesic circles $X$ and $Y$.  From the Automorphism Theorem, the only way $\b$ can act on $V \cup W$ is if $\b(V) = W$, and the action of $\b$ partitions $V \cup W$ into $s$-cycles along with one 4-cycle.  Since any point in $S^3-(X\cup Y)$ has an orbit of size $rs > 8$ under the action of $H$, we must have four vertices embedded on one of $X$ and $Y$.  Suppose the four vertices (two from each of $V$ and $W$) are embedded on $X$; then $\b$ consists of a combination of a rotation around $X$ with a rotation of order 4 around $Y$.

First suppose that $r$ is even.  Since $2n = 2ks + s + 4$ for some integer $k$, we must embed at least $s+4 \geq 12$ vertices on $X\cup Y$; since $\b(V) = W$, these vertices are evenly divided between $V$ and $W$.  By the same argument used in Lemma \ref{L:s+2mod2sr=4}, this means we must embed the vertices of $V$ on one circle and the vertices of $W$ on the other.  But this is impossible, since $\b$ interchanges the vertex sets.

If $r$ is odd, then $\a(V) = V$.  By the Automorphism Theorem, $\a$ must fix two vertices of each of $V$ and $W$, or have a 2-cycle in each vertex set.  Since $r$ is odd, $\a$ does not have any 2-cycles, so it must fix two vertices in each set.  Hence $\a$ is a rotation of order $r$ about $X$, and fixes the four vertices we have embedded on $X$.  Since $4 \vert s$ and $r$ is odd, $s = 4qr$ for some integer $q$.  Then $\b^{4q}$ is a rotation of order $r$ around $X$, so $\b^{4q} = \a^{\pm 1}$.  But then $H$ is cyclic, which is a contradiction.

So the embedding $\Gamma$ does not exist.
\end{proof}\\

\section{Proof of Theorem \ref{T:product}} \label{S:construct}

Now we will prove that for each condition on $n$ listed in Theorem \ref{T:product} there does exist embeddings $\Gamma_1$ and $\Gamma_2$ of $K_{n,n}$ such that $\TSG(\Gamma_1) = \Z_r \x \Z_s$ and $\TSG(\Gamma_2) = (\Z_r \x \Z_s)\ltimes \Z_2$, thus proving Theorem \ref{T:product}. 

We will first show that for each condition on $n$ listed in Theorem \ref{T:product} there exists an embedding of $K_{n,n}$ with topological symmetry group containing $(\Z_r \x \Z_s)\ltimes \Z_2$.  Then we will use the Subgroup Lemma and Subgroup Corollary to modify these embeddings (when possible) so that the topological symmetry group is isomorphic to $\Z_r \x \Z_s$ or $(\Z_r \x \Z_s)\ltimes \Z_2$. When we construct our embeddings we will use the following subgroups of $\so(4)$.  As in Section \ref{S:cyclic}, let $A$ be a plane in $\R^4$ and $B$ be its orthogonal complement, and let $C$ be a plane spanned by a vector in $A$ and a vector in $B$.  We will let $X$, $Y$ and $Z$ denotes the intersections with $S^3$ of planes $A$, $B$ and $C$, respectively.

\begin{itemize}

\item Let $g$ be a rotation of order $r$ about $A$ and $h$ be a rotation of order $s$ about $B$. Let $\varphi$ be a rotation of order $2$ about $C$. Then $J_1=\langle g,h, \varphi \rangle\cong(\mathbb{Z}_{r}\times\mathbb{Z}_{s})\ltimes\mathbb{Z}_{2}$. 

\item Suppose that $4 \vert s$.  Let $g$ be a rotation of order $2$ around $A$. Let $h$ be a glide rotation which is the product of a rotation of order $4$ around $A$ and a rotation of order $s$ around $B$. Therefore $h$ has order $\lcm(4,s) = s$. Let $\varphi$ be a rotation of order $2$ about $C$. Then $J_2=\langle g,h, \varphi \rangle\cong(\mathbb{Z}_{2}\times\mathbb{Z}_{s})\ltimes\mathbb{Z}_{2}$. 

\end{itemize} \medskip

\begin{lemma} \label{L:n=0}
If $n \equiv 0 \pmod{s}$ and $H = \Z_r \x \Z_s$ or $(\Z_r \x \Z_s) \ltimes \Z_2$, then there exists an embedding, $\Gamma$, of $K_{n,n}$ in $S^{3}$ such that $H \subseteq \TSG(\Gamma)$.  Moreover, if $n \geq 2rs$, we can choose the embedding so that $H = \TSG(\Gamma)$.
\end{lemma}
\begin{proof}
Since $n\equiv 0 \pmod {s}$, $n = 2krs + ls$ for some integers $k$ and $l$, where $0 \leq l < 2r$. We will use the group of motions $J_{1}$. Pick a small ball, $M$, such that for each non-trivial $h \in J_1$, $h(M) \cap M = \emptyset$ (i.e. $J_1$ acts freely on $M$). We will pick $k$ points, $p_{1}, ... , p_{k}$, and $k$ points, $q_{1}, ..., q_{k}$, inside $M$. Then the orbit of each point has $2rs$ elements. We will embed $2krs$ vertices of $V$ as the points in the orbits of the $p_{i}'s$ and $2krs$ vertices of $W$ as the points in the orbits of the $q_{i}'s$. We still have $ls$ vertices from $V$ and $ls$ vertices from $W$ to embed. Let $F$ denote the union of all the axes of the rotations $g^ah^b\vf$.  If $l$ is even, place $\frac{l}{2}$ points on $X - F$ such that each point has a distinct orbit under the action of $J_1$. Then the orbit of each point has $2s$ points. Embed the $ls$ vertices of $V$ as the points in these orbits. Since $r|s$ then $s=rm$ for some $m\in \mathbb{Z}$. Place $\frac{lm}{2}$ points on $Y - F$ such that each point has a distinct orbit under $J_1$. Then each orbit has $2r$ points. Embed the $ls=lmr$ vertices of $W$ as the points in these orbits.

If $l = 2j+1$ is odd, embed $2js$ vertices of $V$ and $W$ as described above.  We are left with $s$ vertices from each set.  Let $x$ be one of the two points in $X \cap Z$; then the orbit of $x$ under $J_1$ has $s$ points.  Embed the remaining vertices of $V$ on $X$ as the points in the orbit of $x$ (if $s$ is even, there will be vertices embedded at both points of $X \cap Z$).  Since $r|s$, $s=rm$ for some integer $m$. If $m$ is even, place $\frac{m}{2}$ points on $Y - F$. Then the orbit of each point has $2r$ elements. Embed the $s=rm$ vertices of $W$ as the points in the orbits. If $m$ is odd, then $m=2t+1$ for some integer $t$. Embed the $2tr$ vertices as in the case when $m$ is even. There are $r$ vertices remaining. Let $y$ be one of the two points in $Y \cap Z$; then the orbit of $y$ under $J_1$ has $r$ points. Embed the remaining $r$ vertices of $W$ as the points in the orbit of $y$ (if $r$ is even, there will be vertices embedded at both points of $Y \cap Z$).

Now we will show that we can embed the edges of $K_{n,n}$. If we have not embedded vertices in $X \cap Z$ and $Y\cap Z$, then no element of $J_1$ fixes an adjacent pair of vertices, and conditions (1), (2) and (3) of the Edge Embedding Lemma are satisfied.  However, if there is a point $v \in V$ on $X\cap Z$ and a point $w\in W$ on $Y \cap Z$, then each pair $\{h^i(v), g^j(w)\}$ is fixed by $h^{2i}g^{2j}\varphi$. Since $h^{2i}g^{2j}\varphi$ is the only element of $J_1$ fixing this pair, condition (1) of the Edge Embedding Lemma is met. Since at most two vertices of $V$ and two vertices of $W$ are embedded on each circle $h^ig^j(Z)$, the vertices from $V$ and $W$ alternate around the circle, and the circles intersect only on $X$ and $Y$, there exist arcs bounded by each pair whose interiors are disjoint from $V\cup W$ and from each other. Thus condition (2) is met. The only motion of $J_1$ setwise fixing the pair of vertices $\{h^i(v), g^j(w)\}$, or any point on the interior of the arc between them, is the rotation $h^{2i}g^{2j}\varphi$. Hence condition (3) is met. Since no adjacent vertices are interchanged by any motion of $J_1$, conditions (4) and (5) are met. Therefore we are able to embed the edges of $K_{n,n}$ to get an embedding $\Gamma$ so that $(\mathbb{Z}_{r}\times\mathbb{Z}_{s})\ltimes\mathbb{Z}_{2}\subseteq \TSG(\Gamma)$. 

Now we will apply the Subgroup Lemma to show that we can modify the embedding so that $\TSG = (\mathbb{Z}_{r}\times\mathbb{Z}_{s})\ltimes \Z_2$ if $k > 0$ (i.e. $n \geq 2rs$). Since $k > 0$, we have a $2rs$-orbit from $V$ and and $2rs$-orbit from $W$, say $\{v_{1}, v_{2}, \ldots, v_{2rs}\}$ and $\{w_{1}, w_{2}, \ldots, w_{2rs}\}$ respectively.  We label the vertices so that for any $j \in J_1$, if $v_i = j(v_1)$, then $w_i = j(w_1)$. Let $e_{i} = \overline{v_1w_i}$. Note that the orbits $\langle e_{i} \rangle_{J_{1}}$ are all distinct. Suppose $\psi$ is an automorphism of $K_{n,n}$ which fixes $e_{1}$ pointwise, and fixes each $\langle e_{i} \rangle_{J_{1}}$ setwise. This means that $v_{1}$ and $w_{1}$ are both fixed pointwise. Since $e_{i}$ is the only edge in $\langle e_{i} \rangle_{J_{1}}$ that is adjacent to $v_{1}$, then $e_{i}$ is also fixed. This implies that $w_{i}$ is fixed for every $i$. Thus a subgraph $K_{1,2rs}$ is fixed pointwise and since $s \geq 3$, $K_{1,2rs}$ cannot be embedded in $S^{1}$. So by the Subgroup Lemma there is an embedding $\Gamma'$ such that $\TSG(\Gamma') = (\mathbb{Z}_{r}\times\mathbb{Z}_{s})\ltimes \Z_2$.  Moreover, since $e_1$ was not fixed by any non-trivial element of $J_1$, the Subgroup Corollary implies there is another embedding $\Gamma''$ such that $\TSG(\Gamma'') = \mathbb{Z}_{r}\times\mathbb{Z}_{s}$.
\end{proof}
\medskip

\begin{lemma} \label{L:n=2,r=2}
If $n \equiv 2 \pmod{2s}$, $2 \vert s$ and $H = \Z_2 \x \Z_s$ or $(\Z_2 \x \Z_s) \ltimes \Z_2$, then there exists an embedding, $\Gamma$, of $K_{n,n}$ in $S^{3}$ such that $\TSG(\Gamma) = H$.  
\end{lemma}
\begin{proof}
In this case $n = 2ks + 2$ for some integer $k$, so $2n = 4ks + 4$. Pick a small ball, $M$, such that for each non-trivial $h \in J_1$, $h(M) \cap M = \emptyset$ (i.e. $J_1$ acts freely on $M$).  Pick $k$ points $p_1, \dots, p_k$ inside $M$; the orbit of each $p_i$ under $J_i$ contains $4s$ points.  Embed vertices of $V$ at each $g^ah^b(p_i)$ and vertices of $W$ at each $g^ah^b\vf(p_i)$.  Then there are four remaining vertices $v_{1},v_{2} \in V$ and $w_{1},w_{2} \in W$. Embed $v_1, w_1, v_2, w_2$ in order around $Y$, equally spaced, at the points $\pi/4$ radians away from $Y \cap Z$.  Then $g(v_1) = v_2$, $g(w_1) = w_2$, $\vf(v_1) = w_2$ and $\vf(w_1) = v_2$.

Since pairs of adjacent vertices are only fixed by rotations around $Y$, and all rotations around $Y$ have the same fixed point set, condition (1) of the Edge Embedding Lemma is met. Since the four vertices on $Y$ alternate $v$'s and $w$'s, there exists an arc bounded by each adjacent pair whose interior is disjoint from $V \cup W$ and any other such arc. So condition (2) is met. Also a pair of vertices bounding some arc is setwise invariant only under a rotation around $Y$. Since rotations around $Y$ fix the arc setwise, then condition (3) is met. Lastly no adjacent pair of vertices is interchanged by non-trivial elements of $J_1$, so conditions (4) and (5) are met. Therefore we are able to embed the edges of $K_{n,n}$ to get an embedding $\Gamma$ so that $(\mathbb{Z}_{2}\times\mathbb{Z}_{s})\ltimes\mathbb{Z}_{2}\subseteq \TSG(\Gamma)$. 

Now we will apply the Subgroup Lemma to show that we can modify the embedding so that $\TSG = (\mathbb{Z}_{2}\times\mathbb{Z}_{s})\ltimes \Z_2$. Since $n \geq 3$, there is at least one orbit of $4s$ vertices embedded in $S^3$ so that none of the vertices is fixed by any element of $J_1$.  Let $v \in V$ be one of these vertices, so $g^ah^b(v)$ is in $V$ and $g^ah^b\vf(v)$ is in $W$.  Let $e_i = \overline{v\, h^i\vf(v)}$.  Observe that all the $e_i$'s have distinct orbits under $J_1$. Suppose $\psi$ is an automorphism of $K_{n,n}$ which fixes $e_{0}$ pointwise, and fixes $\langle e_{i} \rangle_{J_{1}}$ setwise. This means that $v$ and $\vf(v)$ are both fixed pointwise. Since $e_{i}$ is the only edge in $\langle e_{i} \rangle_{J_{1}}$ that is adjacent to $v$, then $e_{i}$ is also fixed. This implies that $h^i\vf(v)$ is fixed for every $i$. Thus a subgraph $K_{1,s}$ is fixed pointwise and since $s \geq 3$, $K_{1,s}$ cannot be embedded in $S^{1}$. So by the Subgroup Lemma there is an embedding $\Gamma'$ such that $\TSG(\Gamma') = (\mathbb{Z}_{2}\times\mathbb{Z}_{s})\ltimes \Z_2$.  Moreover, since $e_1$ is not fixed by any non-trivial element of $J_1$, the Subgroup Corollary implies there is also an embedding $\Gamma''$ such that $\TSG(\Gamma'') = \Z_2 \x \Z_s$.
\end{proof}
\medskip

\begin{lemma} \label{L:n=s+2,r=2}
If $n \equiv s+2 \pmod{2s}$, $4\vert s$ and $H = \Z_2 \x \Z_s$ or $(\Z_2 \x \Z_s) \ltimes \Z_2$, then there exists an embedding, $\Gamma$, of $K_{n,n}$ in $S^{3}$ such that $H \subseteq \TSG(\Gamma)$.  Moreover, except in the case when $n=6$, $s=4$ and $H = (\Z_2 \x \Z_4) \ltimes \Z_2$, we can choose the embedding so that $H = \TSG(\Gamma)$.
\end{lemma}
\begin{proof}
In this case $n = 2ks + s + 2$ for some integer $k$, so $2n = 4ks + 2s +4$.  We will use the group of motions $J_2$.  Pick a small ball, $M$, such that for each non-trivial $h \in J_2$, $h(M) \cap M = \emptyset$ (i.e. $J_2$ acts freely on $M$).  Pick $k$ points $p_1, \dots, p_k$ inside $M$; the orbit of each $p_i$ under $J_2$ contains $4s$ points.  Embed vertices of $V$ at each $g^ah^b\vf^c(p_i)$ where $b$ is even and vertices of $W$ at each $g^ah^b\vf^c(p_i)$ where $b$ is odd.  So $g(V) = V$, $h(V) = W$ and $\vf(V) = V$.  We are left with $2s+4$ vertices to embed. First consider four points $v_{1}, v_{2} \in V$ and $w_{1}, w_{2} \in W$. Embed $v_{1}$ and $v_{2} = h^2(v_1)$ at the two points in $Y\cap Z$. Embed $w_{1}$ on $Y$ as $h(v_1)$. Embed $w_{2}$ as $g(w_1) = h^3(v_1)$. There are $2s$ remaining vertices of $V \cup W$ to embed. Choose a point $q$ on $Z-(X \cup Y)$. Since $g$ has order 2 and $h$ has order $s$, the image of $Z$ under $g$ and $h$ will be $\frac{s}{2}$ distinct circles, with four images of $q$ on each circle. Since $v_{1}$ and $v_{2}$ are on $Z$ embed a vertex of $V$ at $q$. Then the three images of $q$ on $Z$ are $g(q)$, $h^{\frac{s}{2}}(q)$ and $gh^{\frac{s}{2}}(q)$. 

Since $g$ fixes $V$ and, since $\frac{s}{2}$ is even, $h^{\frac{s}{2}}(V)=V$, all vertices on $Z$ are in $V$. Following in this manner, embed the $2s$ vertices of $V \cup W$ as the orbit of $q$ such that $g^ah^{2k}(q)\in V$ and $g^ah^{2k+1}(q)\in W$ for $k\in \mathbb{Z}$. Thus images of $Z$ either contain only vertices of $V$ or only vertices of $W$. Since only rotations around $Y$ fix pairs of adjacent vertices, and all rotations around $Y$ have the same fixed point set, condition (1) of the Edge Embedding Lemma is met. Since  $v_{1},v_{2}, w_{1}$ and $w_{2}$ are the only vertices embedded on $Y$, alternating $v$'s and $w$'s, there is an arc bounded by each pair whose interior is disjoint from $V \cup W$ and any other such arc. Thus condition (2) is met. Also a pair of vertices bounding such an arc is setwise invariant only under a rotation around $Y$. Since rotations around $Y$ fix the arc, then condition (3) is met. Lastly no adjacent pair of vertices is interchanged by non-trivial elements of $J_2$, so conditions (4) and (5) are met. Therefore we are able to embed the edges of $K_{n,n}$ to get an embedding $\Gamma$ so that $(\mathbb{Z}_{2}\times\mathbb{Z}_{s})\ltimes\mathbb{Z}_{2}\subseteq \TSG(\Gamma)$. 

Now we will apply the Subgroup Lemma to show that we can modify the embedding so that $\TSG = (\mathbb{Z}_{2}\times\mathbb{Z}_{s})\ltimes \Z_2$. If $k > 0$, at least $4s$ vertices are embedded in the complement of all the fixed point sets of elements of $J_2$; then we can find an embedding $\Gamma'$ with $\TSG(\Gamma') = (\mathbb{Z}_{2}\times\mathbb{Z}_{s})\ltimes \Z_2$ as we did in Lemma \ref{L:n=2,r=2}.  Also as in Lemma \ref{L:n=2,r=2}, there is an edge not fixed by any nontrivial element of the topological symmetry group, so the Subgroup Corollary implies there is an embedding $\Gamma''$ with $\TSG(\Gamma'') = \Z_2 \x \Z_s$.  

If $k = 0$, then let $v$ be a vertex embedded on $Z - (X \cup Y)$.  Let $e_1 = \overline{v h(v)}$, so $\langle e_1\rangle_{J_2} = \{\overline{h^i(v)h^{i+1}(v)}, \overline{gh^i(v) gh^{i+1}(v)}\}$. Suppose $\psi$ is an automorphism of $K_{n,n}$ which fixes $e_{1}$ pointwise, and fixes $\langle e_{1} \rangle_{J_{2}}$ setwise.  So $v$ and $h(v)$ are both fixed.  Notice that the only other edge in the orbit of $e_1$ which is adjacent to $h(v)$ is $\overline{h(v)h^2(v)}$.  Since $h(v)$ and $e_1$ are both fixed, this means $\overline{h(v)h^2(v)}$ is fixed, and hence $h^2(v)$ is fixed. Proceeding inductively, we can show that $h^i(v)$ is fixed for every $i$.  This a subgraph $K_{s/2,s/2}$ is fixed pointwise.  If $s > 4$, this subgraph cannot be embedded in $S^1$, so there is an embedding $\Gamma'$ with $\TSG(\Gamma') = (\mathbb{Z}_{2}\times\mathbb{Z}_{s})\ltimes \Z_2$.  Also, $e_1$ is not fixed by any nontrivial element of $J_2$, so the Subgroup Corollary implies there is an embedding $\Gamma''$ with $\TSG(\Gamma'') = \Z_2 \x \Z_s$.

We are left with the case when $s = 4$ and $n = 6$.  In this case we can only show there is an embedding with $\TSG = \Z_2 \x \Z_s$.  We first embed the vertices as described before, but now view them as acted on only by the subgroup $G = \Z_2 \x \Z_s$ of $J_2$ generated by $g$ and $h$.  So there are 4 vertices embedded on $Y$ and 8 embedded in $S^3 - (X \cup Y)$.  Let $v$ be a vertex of $V$ embedded in $S^3 - (X\cup Y)$, and let $e_1 = \overline{v h(v)}$.  Also let $e_2 = \overline{v w_1}$, where $w_1$ is a vertex of $W$ embedded on $Y$.  Suppose $\psi$ is an automorphism of $K_{n,n}$ which fixes $e_{1}$ pointwise, and fixes $\langle e_{1} \rangle_{G}$ and $\langle e_{2} \rangle_{G}$ setwise.  The orbit of $e_1$ under $G$ is the same as its orbit under $J_2$, so by the same argument as in the last paragraph $\psi$ must fix each vertex $h^i(v)$, and so fixes a subgraph $K_{2,2}$.  However, under the action of $G$, $e_2$ is the only edge in its orbit which is adjacent to $v$ (as opposed to the action of $J_2$, where $\overline{v w_2}$ is also in the orbit).  So $e_2$, and therefore $w_1$, is also fixed by $\psi$.  Hence $\psi$ fixes a subgraph $K_{3,2}$ which cannot be embedded in $S^1$.  So there is an embedding $\Gamma'$ of $K_{6,6}$ with $\TSG(\Gamma') = \mathbb{Z}_{2}\times\mathbb{Z}_{4}$.
\end{proof}
\medskip

\begin{lemma} \label{L:n=2,r=4}
If $n\equiv 2 \pmod {2s}$, $4\vert s$ and $H = \Z_4 \x \Z_s$ or $(\Z_4 \x \Z_s) \ltimes \Z_2$, then there exists an embedding, $\Gamma$, of $K_{n,n}$ in $S^{3}$ such that $H \subseteq \TSG(\Gamma)$. Moreover, except in the case when $n=10$, $s=4$ and $H = (\Z_4 \x \Z_4) \ltimes \Z_2$, we can choose the embedding so that $H = \TSG(\Gamma)$.
\end{lemma}
\begin{proof}
We need to consider when $n \equiv 2 \pmod{4s}$ and when $n \equiv 2s+2 \pmod{4s}$.  In this case $n = 4ks + 2$ or $4ks+2s+2$ for some integer $k$, so $2n = 8ks + 4$ or $8ks+4s+4$.  We will use the group of motions $J_1$, with $r = 4$.  Pick a small ball, $M$, such that for each non-trivial $h \in J_1$, $h(M) \cap M = \emptyset$ (i.e. $J_1$ acts freely on $M$).  Pick $k$ points $p_1, \dots, p_k$ inside $M$; the orbit of each $p_i$ under $J_1$ contains $8s$ points.  Embed vertices of $V$ at each $g^ah^b\vf^c(p_i)$ where $a$ is even and vertices of $W$ at each $g^ah^b\vf^c(p_i)$ where $a$ is odd.  So $g(V) = W$, $h(V) = V$ and $\vf(V) = V$.  We are left with either $4$ or $4s+4$ vertices to embed.  Consider the four vertices $v_1, v_2, w_1, w_2$.  Embed $v_1$ at one point of $Y \cap Z$ and let $g(v_1) = w_1$, $g^2(v_1) = v_2$ and $g^3(v_1) = w_2$.  If there are an additional $4s$ vertices, we embed them as follows.  Let $z$ be a point of $Z - (X \cup Y)$; then the orbit of $z$ under $J_1$ has $4s$ elements.  Embed vertices of $V$ at $g^ah^b(z)$ where $a$ is even, and vertices of $W$ at $g^ah^b(z)$ where $a$ is odd.  Then each image of $Z$ contains six vertices, all from $V$ or all from $W$; circles $g^ah^b(Z)$ where $a$ is even contain $v_1$ and $v_2$, and circles $g^ah^b(Z)$ where $a$ is odd contain $w_1$ and $w_2$.  

The only pairs of adjacent vertices fixed by an element of $J_1$ are the points embedded on $Y$, which are only fixed by rotations about $Y$, so condition (1) of the Edge Embedding Lemma is met.  Since  $v_{1},v_{2}, w_{1}$ and $w_{2}$ are the only vertices embedded on $Y$, alternating $v$'s and $w$'s, there is an arc bounded by each pair whose interior is disjoint from $V \cup W$ and any other such arc. Thus condition (2) is met. Also a pair of vertices bounding such an arc is setwise invariant only under a rotation around $Y$. Since all rotations around $Y$ fix the arc pointwise, then condition (3) is met. Lastly no adjacent pair of vertices is interchanged by non-trivial elements of $J_1$, so conditions (4) and (5) are met. Therefore we are able to embed the edges of $K_{n,n}$ to get an embedding $\Gamma$ so that $(\mathbb{Z}_{4}\times\mathbb{Z}_{s})\ltimes\mathbb{Z}_{2}\subseteq \TSG(\Gamma)$.

Now we will apply the Subgroup Lemma to show that we can modify the embedding so that $\TSG = (\mathbb{Z}_{4}\times\mathbb{Z}_{s})\ltimes \Z_2$. If $k > 0$, at least $8s$ vertices are embedded in the complement of all the fixed point sets of elements of $J_1$; then we can find embeddings $\Gamma'$ and $\Gamma''$ with $\TSG(\Gamma') = (\mathbb{Z}_{4}\times\mathbb{Z}_{s})\ltimes \Z_2$ and $\TSG(\Gamma'') = \Z_4 \x \Z_s$ as we did in Lemma \ref{L:n=2,r=2}.  If $k = 0$, then let $v$ be a vertex of $V$ embedded on $Z - (X \cup Y)$.  Let $e_1 = \overline{v gh(v)}$, so $\langle e_{1} \rangle_{J_{1}} = \{\overline{g^ih^j(v) g^{i+1}h^{j+1}(v)}\}$ (notice that $\vf(\overline{v gh(v)}) = \overline{v g^{-1}h^{-1}(v)}$, which is still in the set). Suppose $\psi$ is an automorphism of $K_{n,n}$ which fixes $e_{1}$ pointwise, and fixes $\langle e_{1} \rangle_{J_{1}}$ setwise.  So $v$ and $gh(v)$ are both fixed.  However, the only edges in the orbit of $e_1$ adjacent to $gh(v)$ are $\overline{v gh(v)}$ and $\overline{gh(v) g^2h^2(v)}$.  Since $v$ and $gh(v)$ are both fixed, this means $g^2h^2(v)$ is also fixed.  Continuing inductively, every $g^ih^i(v)$ is fixed.  The points $\{g^ih^i(v)\}$, where $0 \leq i < s$, alternate between vertices of $V$ and $W$; so these points induce a subgraph of $\Gamma$ isomorphic to $K_{s/2,s/2}$ that is fixed pointwise.  If $s > 4$, this subgraph cannot be embedded in $S^1$, so there is an embedding $\Gamma'$ with $\TSG(\Gamma') = (\mathbb{Z}_{4}\times\mathbb{Z}_{s})\ltimes \Z_2$.  Also, $e_1$ is not fixed pointwise by any nontrivial element of $J_2$, so the Subgroup Corollary implies there is an embedding $\Gamma''$ with $\TSG(\Gamma'') = \Z_4 \x \Z_s$.

We are left with the case when $s = 4$ and $n = 10$.  In this case we can only show there is an embedding with $\TSG = \Z_4 \x \Z_s$.  We first embed the vertices as described before, but now view them as acted on only by the subgroup $G = \Z_4 \x \Z_s$ of $J_1$ generated by $g$ and $h$.  So there are 4 vertices embedded on $Y$ and 16 embedded in $S^3 - (X \cup Y)$.  Let $v$ be a vertex ov $V$ embedded in $S^3 - (X\cup Y)$, and let $e_1 = \overline{v gh(v)}$.  Also let $e_2 = \overline{v w_1}$, where $w_1$ is a vertex of $W$ embedded on $Y$.  Suppose $\psi$ is an automorphism of $K_{n,n}$ which fixes $e_{1}$ pointwise, and fixes $\langle e_{1} \rangle_{G}$ and $\langle e_{2} \rangle_{G}$ setwise.  The orbit of $e_1$ under $G$ is the same as its orbit under $J_1$, so by the same argument as in the last paragraph $\psi$ must fix each vertex $g^ih^i(v)$, and so fixes a subgraph $K_{2,2}$.  However, under the action of $G$, $e_2$ is the only edge in its orbit which is adjacent to $v$ (as opposed to the action of $J_1$, where $\overline{v w_2}$ is also in the orbit).  So $e_2$, and therefore $w_1$, is also fixed by $\psi$.  Hence $\psi$ fixes a subgraph $K_{3,2}$ which cannot be embedded in $S^1$.  So there is an embedding $\Gamma'$ of $K_{10,10}$ with $\TSG(\Gamma') = \mathbb{Z}_{4}\times\mathbb{Z}_{4}$.
\end{proof}
\medskip

Combining the results of Sections \ref{S:necessity} and \ref{S:construct} gives us the proof of Theorem \ref{T:product}.

\begin{them2}
Let $n>2$. There exists an embedding, $\Gamma$, of $K_{n,n}$ in $S^{3}$ such that $H \subseteq\TSG(\Gamma)$ for $H = \mathbb{Z}_{r}\times\mathbb{Z}_{s}$ or $(\mathbb{Z}_{r}\times\mathbb{Z}_{s})\ltimes\mathbb{Z}_{2}$, where $r \vert s$, if and only if one of the following conditions hold:
	\begin{enumerate}	
	\item $n\equiv 0 \pmod {s}$,
	\item $n\equiv 2 \pmod {2s}$ when $r=2$,
	\item $n\equiv s+2 \pmod {2s}$ when $4\vert s$, and $r=2$,
	\item $n\equiv 2 \pmod {2s}$ when $r=4$.
	\end{enumerate}
Moreover, in each of the above cases, we can construct embeddings $\Gamma$ where $\TSG(\Gamma) = H$ except in the following cases, which are still open: \begin{itemize}
	\item $K_{ls, ls}$, when $1 \leq l < 2r$, $H = \mathbb{Z}_{r}\times\mathbb{Z}_{s}$ or $(\mathbb{Z}_{r}\times\mathbb{Z}_{s})\ltimes\mathbb{Z}_{2}$
	\item $K_{6,6}$, when $H = (\Z_2 \x \Z_4) \ltimes \Z_2$
	\item $K_{10,10}$, when $H = (\Z_4 \x \Z_4) \ltimes \Z_2$
\end{itemize}
\end{them2}

\begin{proof}
First we will show that the conditions are necessary. From Lemma \ref{L:cyclic1}, we know that we have the following constrictions on $n$: $n \equiv 0,1,2 \pmod {s}$, $n \equiv 0 \pmod {\frac{s}{2}}$ and $s$ even or $n \equiv 2\pmod{ \frac{s}{2}}$ and $4|s$. By Lemma \ref{L:1mods} we can eliminate the case when $n \equiv 1 \pmod{s}$.  By Lemma \ref{L:2mods2}, we can only have $n \equiv 2\pmod{s}$ if $r = 2$ or $4$.  If $r = 2$ and $n \equiv s+2 \pmod{2s}$, then by Lemma \ref{L:s+2mod2sr=2} we must have $4 \vert s$.  If $r = 4$, then by Lemma \ref{L:s+2mod2sr=4} we cannot have $n \equiv s+2 \pmod{2s}$.  Finally the cases when $n \equiv 0 \pmod {\frac{s}{2}}$ or $n \equiv 2\pmod{ \frac{s}{2}}$ (that are not covered by other cases) are ruled out by Lemmas \ref{L:s/2mods} and \ref{L:s/2+2mods}, respectively. Thus we have shown the necessity of the conditions.

Lemmas \ref{L:n=0}, \ref{L:n=2,r=2}, \ref{L:n=s+2,r=2} and \ref{L:n=2,r=4} show that we can find embeddings to realize each of the remaining conditions, and that we can achieve equality in all cases except the three exceptions noted.
\end{proof}


\small

\normalsize

\end{document}